\documentclass[12pt]{amsart}
\usepackage{amscd, amssymb, pgf, tikz, hyperref, mathabx}

\newcommand{\field}[1]{\mathbb{#1}}
\newcommand{\CC}{\field{C}}
\newcommand{\NN}{\field{N}}

\newcommand{\TT}{\field{T}}
\newcommand{\ZZ}{\field{Z}}

\newcommand{\Aa}{\mathcal{A}}

\newcommand{\Ff}{\mathcal{F}}
\newcommand{\Gg}{\mathcal{G}}
\newcommand{\Hh}{\mathcal{H}}
\newcommand{\Kk}{\mathcal{K}}
\newcommand{\Ll}{\mathcal{L}}
\newcommand{\Mm}{\mathcal{M}}
\newcommand{\Nn}{\mathcal{N}}
\newcommand{\Oo}{\mathcal{O}}

\newcommand{\Tt}{\mathcal{T}}
\newcommand{\Ss}{\mathcal{S}}

\newcommand{\Xx}{\mathcal{X}}

\newcommand{\la}{\langle}
\newcommand{\ra}{\rangle}

\newcommand{\ds}{\displaystyle}

\newtheorem{thm}{Theorem}[section]
\newtheorem{cor}[thm]{Corollary}
\newtheorem{lem}[thm]{Lemma}
\newtheorem{prop}[thm]{Proposition}

\theoremstyle{definition}
\newtheorem{dfn}[thm]{Definition}

\theoremstyle{remark}
\newtheorem{rmk}[thm]{Remark}

\newtheorem{example}[thm]{Example}

\newtheorem*{examples*}{Examples}

\numberwithin{equation}{subsection}
\numberwithin{equation}{subsection}

\title{ Higman-Thompson groups from self-similar groupoid actions }

\author{Valentin Deaconu}
\address{Valentin Deaconu \\ Department of Mathematics (084)\\ University
of Nevada\\ Reno NV 89557-0084\\ USA} \email{vdeaconu@unr.edu}



\begin{document}
\begin{abstract}
Given a self-similar groupoid action $(G,E)$ on a finite directed graph, we prove some properties of the corresponding ample groupoid of germs $\Gg(G,E)$. We study  the analogue of the Higman-Thompson group associated to $(G,E)$ using $G$-tables and relate it to the topological full group of $\Gg(G,E)$, which is isomorphic to a subgroup of unitaries in the algebra $C^*(G,E)$. 
After recalling some  concepts in groupoid homology, we discuss the Matui's AH-conjecture for $\Gg(G,E)$ in some particular cases.

\end{abstract}
\maketitle
\section{introduction}

Self-similar group actions are defined using subgroups of the automorphism group of a rooted tree which is  viewed as the path space of a graph with one vertex and $n$ edges. Given a finite directed graph $E$ with no sources, the corresponding path space gives rise to a union of trees (or forest) $T_E$, and it is natural to consider self-similar actions of subgroupoids $G$ of PIso$(T_E)$, the set of partial isomorphsims of $T_E$. 

Given such a self-similar action $(G,E)$, we study the properties of the generalized Exel-Pardo  ample groupoid $\Gg(G,E)$. Since the source map of the graph is in general not $G$-equivariant, this  determines certain modifications in the proofs  of similar results from \cite{EP}, where $G$ is a group acting on $E$ by graph automorphisms. We study the topological full group $\ldbrack\Gg(G,E)\rdbrack$ and we  relate it  to the Higman-Thompson type group defined by $G$-tables for the action $(G,E)$. 
After recalling some results  about groupoid homology, we discuss the Matui's AH-conjecture for $\Gg(G,E)$ and related issues.

We begin  by recalling the definition of a self-similar groupoid action $(G,E)$ from \cite{LRRW} and give several examples, illustrating some differences with the case when $G$ is a group acting by automorphisms of the graph $E$. We  recall some facts about the structure of the associated $C^*$-algebra  $C^*(G,E)$ defined using a $C^*$-correspondence over $C^*(G)$ and show that it contains a copy of the graph algebra $C^*(E)$. We review the definition of the ample groupoid of germs $\Gg(G,E)$, inspired from the Exel-Pardo groupoid from \cite{EP} and describe when this groupoid is minimal and effective. 

We then define the Higman-Thompson group $V_E(G)$ associated to $(G,E)$ using $G$-tables and describe a faithful unitary representation of $V_E(G)$ in the $C^*$-algebra $C^*(G,E)$. Topological full groups associated to dynamical systems and to \' etale groupoids are complete invariants for continuous orbit equivalence and groupoid isomorphism. They  provide means of constructing new countable groups with interesting properties. 

After recalling the definion of homology of \' etale groupoids  introduced by Crainic and Moerdijk in \cite{CM} and Matui's AH-conjecture, we specialize to the groupoid $\Gg(G,E)$ associated to a self-similar action $(G, E)$ and its topological full group $\ldbrack\Gg(G,E)\rdbrack$ which is isomorphic to $V_E(G)$. The kernel $\Hh(G,E)$ of the groupoid cocycle $\rho:\Gg(G,E)\to \ZZ$ plays an important r\^ole  in the computation of the groupoid homology $H_i(\Gg(G,E))$, particularly when $G$ is transitive. We prove that in some cases, the index map $I_\Hh:\ldbrack\Hh(G,E)\rdbrack\to H_1(\Hh(G,E))$ is surjective. We also show that in a particular example, the kernel of the index map $\ldbrack\Gg(G,E)\rdbrack\to H_1(\Gg(G,E))$ is not generated by transpositions.

\bigskip

\bigskip

\section{Self-similar actions of groupoids on graphs}
\bigskip

Let $E = (E^0, E^1, r, s)$ be a finite directed graph  with no sources. For $k \ge 2$, define  the set of paths of length $k$ in $E$ as
\[E^k = \{e_1e_2\cdots e_k : e_i \in E^1,\; r(e_{i+1}) = s(e_i)\}.\]
The maps $r,s$ are naturally extended to $E^k$ by taking
\[r(e_1e_2\cdots e_k)=r(e_1),\;\; s(e_1e_2\cdots e_k)=s(e_k).\]
We denote by $E^* :=\bigcup_{ k\ge 0} E^k$  the space of finite paths (including vertices) and  by $E^\infty$ the infinite path space of $E$ with the usual topology given by the cylinder sets $Z(\alpha)=\{\alpha\xi:\xi\in E^\infty\}$ for $\alpha\in E^*$. 
 Note that $E^\infty$ is a Cantor space precisely when $E$ satisfies condition (L).
 
 We can visualize the set  $E^*$ as indexing the vertices of a union of rooted trees  or forest $T_E$   given by $T_E^0 =E^*$ and with edges \[T_E^1 =\{(\mu, \mu e) :\mu\in E^*, e\in E^1\; \text{and}\; s(\mu)=r(e)\}.\]
 
 \begin{example}\label{ex}
 
 For the graph
 
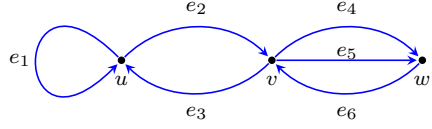
\[\begin{tikzpicture}[shorten >=0.4pt,>=stealth, semithick]
\renewcommand{\ss}{\scriptstyle}
\node[inner sep=1.0pt, circle, fill=black]  (u) at (-2,0) {};
\node[below] at (u.south)  {$\ss u$};
\node[inner sep=1.0pt, circle, fill=black]  (v) at (0,0) {};
\node[below] at (v.south)  {$\ss v$};
\node[inner sep=1.0pt, circle, fill=black]  (w) at (2,0) {};
\node[below] at (w.south)  {$\ss w$};

\draw[->, blue] (u) to [out=45, in=135]  (v);
\node at (-1,0.7){$\ss e_2$};
\draw[->, blue] (v) to [out=-135, in=-45]  (u);
\node at (-1,-0.7) {$\ss e_3$};
\draw[->, blue] (v) to [out=45, in=135]  (w);
\node at (1,0.7){$\ss e_4$};
\draw[->, blue] (v) to (w);
\node at (1,0.1){$\ss e_5$};
\draw[->, blue] (w) to [out=-135, in=-45]  (v);
\node at (1,-0.7) {$\ss e_6$};

\draw[->, blue] (u) .. controls (-3.5,1.5) and (-3.5, -1.5) .. (u);
\node at (-3.35,0) {$\ss e_1$};

\end{tikzpicture}
\]

the forest $T_E$ looks like

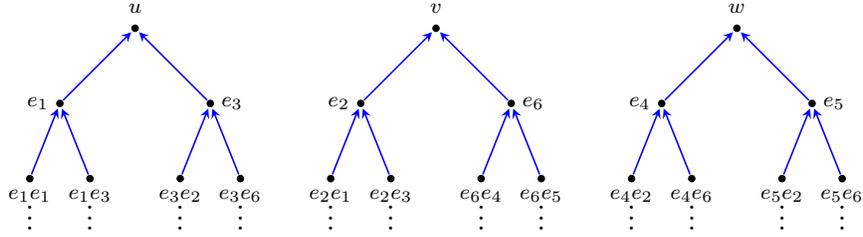
\[
\begin{tikzpicture}[shorten >=0.4pt,>=stealth, semithick]
\renewcommand{\ss}{\scriptstyle}
\node[inner sep=1.0pt, circle, fill=black]  (u) at (-4,4) {};
\node[above] at (u.north)  {$\ss u$};
\node[inner sep=1.0pt, circle, fill=black]  (v) at (0,4) {};
\node[above] at (v.north)  {$\ss v$};
\node[inner sep=1.0pt, circle, fill=black]  (w) at (4,4) {};
\node[above] at (w.north)  {$\ss w$};

\node[inner sep=1.0pt, circle, fill=black] (e1) at (-5,3){};
\node[left] at (-5,3)  {$\ss e_1$};
\node[inner sep=1.0pt, circle, fill=black] (e3) at (-3,3){};
\node[right] at (-3,3)  {$\ss e_3$};
\draw[->, blue] (e1) to  (u);
\draw[-> , blue] (e3) to  (u);
\node[inner sep=1.0pt, circle, fill=black] (e11) at (-5.4,2){};
\node[inner sep=1.0pt, circle, fill=black] (e13) at (-4.6,2){};
\node[below] at (-5.4,2)  {$\ss e_1e_1$};
\node[below] at (-5.4,2) {$\vdots$};
\draw[-> , blue] (e11) to   (e1);
\node[below] at (-4.6,2)  {$\ss e_1e_3$};
\node[below] at (-4.6,2) {$\vdots$};
\draw[-> , blue] (e13) to  (e1);
\node[inner sep=1.0pt, circle, fill=black] (e32) at (-3.4,2){};
\node[inner sep=1.0pt, circle, fill=black] (e36) at (-2.6,2){};
\node[below] at (-3.4,2) {$\ss e_3e_2$};
\node[below] at (-3.4,2) {$\vdots$};
\node[below] at (-2.6,2) {$\ss e_3e_6$};
\node[below] at (-2.6,2) {$\vdots$};
\draw[-> , blue] (e32) to   (e3);
\draw[-> , blue] (e36) to   (e3);

\node[inner sep=1.0pt, circle, fill=black] (e2) at (-1,3){};
\node[left] at (-1,3)  {$\ss e_2$};

\node[inner sep=1.0pt, circle, fill=black] (e6) at (1,3){};
\node[right] at (1,3)  {$\ss e_6$};
\draw[-> , blue] (e2) to  (v);
\draw[-> , blue] (e6) to  (v);
\node[inner sep=1.0pt, circle, fill=black] (e21) at (-1.4,2){};
\node[inner sep=1.0pt, circle, fill=black] (e23) at (-0.6,2){};

\draw[-> , blue] (e21) to   (e2);

\draw[-> , blue] (e23) to  (e2);
\node[inner sep=1.0pt, circle, fill=black] (e64) at (0.6,2){};
\node[inner sep=1.0pt, circle, fill=black] (e65) at (1.4,2){};
\draw[-> , blue] (e64) to   (e6);
\draw[-> , blue] (e65) to   (e6);
\node[below] at (-1.4,2)  {$\ss e_2e_1$};
\node[below] at (-1.4,2) {$\vdots$};
\node[below] at (-0.6,2)  {$\ss e_2e_3$};
\node[below] at (-0.6,2) {$\vdots$};
\node[below] at (0.6,2) {$\ss e_6e_4$};
\node[below] at (0.6,2) {$\vdots$};
\node[below] at (1.4,2) {$\ss e_6e_5$};
\node[below] at (1.4,2) {$\vdots$};

\node[inner sep=1.0pt, circle, fill=black] (e4) at (3,3){};
\node[left] at (3,3)  {$\ss e_4$};

\node[inner sep=1.0pt, circle, fill=black] (e5) at (5,3){};
\node[right] at (5,3)  {$\ss e_5$};
\draw[-> , blue] (e4) to  (w);
\draw[-> , blue] (e5) to  (w);
\node[inner sep=1.0pt, circle, fill=black] (e42) at (2.6,2){};
\node[inner sep=1.0pt, circle, fill=black] (e46) at (3.4,2){};

\draw[-> , blue] (e42) to   (e4);

\draw[-> , blue] (e46) to  (e4);
\node[inner sep=1.0pt, circle, fill=black] (e52) at (4.6,2){};
\node[inner sep=1.0pt, circle, fill=black] (e56) at (5.4,2){};
\draw[-> , blue] (e52) to   (e5);
\draw[-> , blue] (e56) to   (e5);
\node[below] at (2.6,2)  {$\ss e_4e_2$};
\node[below] at (2.6,2) {$\vdots$};
\node[below] at (3.4,2)  {$\ss e_4e_6$};
\node[below] at (3.4,2) {$\vdots$};
\node[below] at (4.6,2) {$\ss e_5e_2$};
\node[below] at (4.6,2) {$\vdots$};
\node[below] at (5.4,2) {$\ss e_5e_6$};
\node[below] at (5.4,2) {$\vdots$};

\end{tikzpicture}
\]

 \end{example}

Recall that a partial isomorphism of the forest $T_E$ corresponding to a given directed graph $E$  consists of a pair $(v, w) \in E^0 \times E^0$ and a bijection $g : vE^* \to wE^*$ such that
\begin{itemize}

\item $g|_{vE^k} : vE^k \to wE^k$ is bijective for all $k\ge 1$.

\item $g(\mu e)\in g(\mu)E^1$ for $\mu\in vE^*$  and $e\in E^1$ with $r(e)=s(\mu)$.
 
 \end{itemize}
 Here $vE^k$ denotes the set of paths $\mu\in E^k$ with $r(\mu)=v$ and similarly for $vE^*$.
The set of partial isomorphisms of $T_E$ forms a discrete groupoid PIso$(T_E)$ with unit space $E^0$. The identity morphisms are  $id_v : vE^* \to vE^*$, the inverse of $g : vE^* \to wE^*$ is $g^{-1} : wE^* \to vE^*$, and the  multiplication is composition. We often identify $v\in E^0$ with $id_v\in$ PIso$(T_E)$. 

In Example \ref{ex}, the sum of the entries in each row of the graph adjacency matrix $\ds A_E=\left(\begin{array}{ccc}1&1&0\\1&0&1\\0&2&0\end{array}\right)$ is the same, so the set of edges ending at each vertex has the same cardinality, but in general it could happen that there is no bijection $g : vE^* \to wE^*$ as above for $v\neq w$, see Example \ref{ka}. In that case, PIso$(T_E)$ is a group bundle. If $E$ has a single vertex, then PIso$(T_E)=$Aut$(T_E)$ is a group.

We are interested in self-similar actions of groupoids on the path space of directed graphs $E$ as introduced and studied in \cite{LRRW}. 

A groupoid $G$ is a  small category with inverses. We will use $d$ and $t$ for the domain and target maps $d,t:G \to G^{(0)}$ to distinguish them from the source and range maps $s,r:E^1\to E^0$ on directed graphs. For $u,v \in G^{(0)}$, we write \[G_u =\{g\in G: d(g)=u\},\;\;G^v =\{g\in G: t(g)=v\},\;\; G_u^v=G_u\cap G^v.\]
The set of composable pairs is denoted $G^{(2)}$.

An \'etale groupoid is a topological groupoid where the target map $t$ (and necessarily the
domain map  $d$) is a local homeomorphism (as a map from $G$ to $G$). The unit space $G^{(0)}$ of an \'etale groupoid is always an open subset of $G$.

\begin{dfn} Let $G$ be an \' etale groupoid. A bisection is an open subset $U\subseteq  G$ such that $d$ and $t$ are both injective when restricted to $U$.
\end{dfn}

Two units $x,y \in G^{(0)}$ belong to the same $G$-orbit if there exists $g \in G$ such that $d(g) = x$ and $t(g) = y$. We denote by $G(x)$ the $G$-orbit of $x$. When every $G$-orbit is dense in $G^{(0)}$, the groupoid $G$ is called minimal.

The isotropy group of a unit $x\in G^{(0)}$ is the group \[G_x^x :=\{g\in G\; | \; d(g)=t(g)=x\},\] and the isotropy bundle is
\[G' :=\{g\in G\; | \; d(g)=t(g)\}= \bigcup_{x\in G^{(0)}} G_x^x.\]
A groupoid $G$ is said to be principal if all isotropy groups are trivial, or equivalently, $G' = G^{(0)}$. We say that $G$ is effective if the interior of $G'$ equals $G^{(0)}$.

\begin{dfn}\label{ss} Let $E$ be a finite directed graph with no sources, and let $G$ be a discrete groupoid with unit space $G^{(0)}=E^0$. A {\em self-similar action} $(G,E)$ of  $G$  on the path space of $E$ is given by a faithful homomorphism $G\to$ PIso$(T_E)$ such that for every $g\in G$ and $e\in d(g)E^1$ there exists a unique $h\in G$ denoted also by $g|_e$ and called the restriction of $g$ to $e$ such that
\[g\cdot(e\mu)=(g\cdot e)(h\cdot \mu)\;\;\text{for all}\;\; \mu\in s(e)E^*.\]

\end{dfn}
\begin{rmk}
We have \[d(g|_e)=s(e),\; t(g|_e)=s(g\cdot e)=g|_e\cdot s(e),\;  r(g\cdot e)=g\cdot r(e).\] This can be visualized as

\[\begin{tikzpicture}[shorten >=0.4pt,>=stealth, thick]
\node[inner sep=1.0pt, circle, fill=black]  (u) at (-1,-1) {};
\node[inner sep=1.0pt, circle, fill=black]  (v) at (1,-1) {};
\node[inner sep=1.0pt, circle, fill=black]  (w) at (-1,1) {};
\node[inner sep=1.0pt, circle, fill=black]  (x) at (1,1) {};
\draw[-> , blue] (v) to   (u);
\draw[-> , red] (u) to   (w);
\draw[-> , red] (v) to   (x);
\draw[-> , blue] (x) to   (w);
\node[below] at (0,-1)  {$e$};
\node[above] at (0,1)  {$g\cdot e$};
\node[left] at (-1,-1)  {$r(e)=d(g)$};
\node[left] at (-1,1)  {$g\cdot r(e)$};
\node[left] at (-1,0)  {$g$};
\node[right] at (1,-1)  {$s(e)$};
\node[right] at (1,1)  {$s(g\cdot e)$};
\node[right] at (1,0)  {$g|_e$};
\end{tikzpicture}
\]

In particular, in general $s(g\cdot e)\neq g\cdot s(e)$, i.e. the source map is not $G$-equivariant as it is assumed for group actions in \cite{EP}. It is shown in Appendix A of \cite{LRRW} that a self-similar group action $(H,E)$ as in \cite{EP} determines a self-similar groupoid action $(H\times E^0, E)$ as in Definition \ref{ss}, where $H\times E^0$ is the action groupoid of the group $H$. But not any self-similar groupoid action comes from a self-similar group action.

It is possible that $g|_e=g$ for all $e\in d(g)E^1$, in which case \[g\cdot(e_1e_2\cdots e_n)=(g\cdot e_1)\cdots(g\cdot e_n).\]
\end{rmk}

 A self-similar groupoid action $(G,E)$  extends to an action of $G$ on the path space $E^*$ and  determines an action of $G$ on the graph $T_E$, in the sense that $G$ acts on both the vertex space $T_E^0$ and the edge space $T_E^1$ and intertwines the range and the source maps of $T_E$, see Definition 4.1 in \cite{De}.
  A self-similar groupoid action also extends to an action of $G$ on $E^\infty$ such that $g\cdot \mu=\eta$ if for all $n$ we have $g\cdot(\mu_1\cdots\mu_n)=\eta_1\cdots \eta_n$.

The faithfulness condition ensures that for each $g \in G$ and $\mu\in E^*$ with $d(g) = r(\mu)$, there is a unique element $g|_\mu\in G$ satisfying 
\[g\cdot(\mu\nu)=(g\cdot\mu)(g|_\mu\cdot \nu)\;\text{ for all}\; \nu\in s(\mu)E^*.\] 
By Proposition 3.6 of \cite{LRRW}, self-similar groupoid actions have the following properties: for $g, h \in G, \mu\in d(g)E^*$, and $\nu\in s(\mu)E^*$,

(1) $g|_{\mu\nu} = (g|_\mu)|_\nu$;

(2) $id_{r(\mu)}|_{\mu} = id_{s(\mu)}$; 

(3) if $(h, g)\in G^{(2)}$, then $(h|_{g\cdot\mu},g|_\mu)\in G^{(2)}$ 
and $(hg)|_\mu = (h|_{g\cdot\mu})(g|_\mu)$; 

(4) $g^{-1}|_\mu=(g|_{g^{-1}\cdot\mu})^{-1}$.
\begin{dfn}\label{trans}
A self-similar action $(G,E)$ is said to be level transitive if the induced action on $E^*$ is transitive on each $E^n$. The action is level transitive iff it is minimal on the infinite path space $E^\infty$. 

\end{dfn}

\begin{example}\label{forest} Let $E$ be the graph from Example \ref{ex}
\[\begin{tikzpicture}[shorten >=0.4pt,>=stealth, semithick]
\renewcommand{\ss}{\scriptstyle}
\node[inner sep=1.0pt, circle, fill=black]  (u) at (-2,0) {};
\node[below] at (u.south)  {$\ss u$};
\node[inner sep=1.0pt, circle, fill=black]  (v) at (0,0) {};
\node[below] at (v.south)  {$\ss v$};
\node[inner sep=1.0pt, circle, fill=black]  (w) at (2,0) {};
\node[below] at (w.south)  {$\ss w$};

\draw[->, blue] (u) to [out=45, in=135]  (v);
\node at (-1,0.7){$\ss e_2$};
\draw[->, blue] (v) to [out=-135, in=-45]  (u);
\node at (-1,-0.7) {$\ss e_3$};
\draw[->, blue] (v) to [out=45, in=135]  (w);
\node at (1,0.7){$\ss e_4$};
\draw[->, blue] (v) to (w);
\node at (1,0.1){$\ss e_5$};
\draw[->, blue] (w) to [out=-135, in=-45]  (v);
\node at (1,-0.7) {$\ss e_6$};

\draw[->, blue] (u) .. controls (-3.5,1.5) and (-3.5, -1.5) .. (u);
\node at (-3.35,0) {$\ss e_1$};

\end{tikzpicture}
\]
with $E^0=\{u,v,w\}, E^1=\{e_1, e_2, e_3, e_4, e_5, e_6\}$. Consider the groupoid $G$ with unit space $G^{(0)}=\{u,v,w\}$ and generators $a,b,c$ such that $d(a)=u, \; t(a)=v=d(b)=t(c), \; t(b)=w=d(c)$ as in the picture

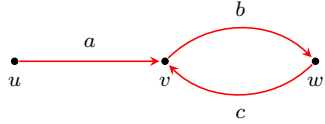
\[\begin{tikzpicture}[shorten >=0.4pt,>=stealth, semithick]
\renewcommand{\ss}{\scriptstyle}
\node[inner sep=1.0pt, circle, fill=black]  (u) at (-2,0) {};
\node[below] at (u.south)  {$\ss u$};
\node[inner sep=1.0pt, circle, fill=black]  (v) at (0,0) {};
\node[below] at (v.south)  {$\ss v$};
\node[inner sep=1.0pt, circle, fill=black]  (w) at (2,0) {};
\node[below] at (w.south)  {$\ss w$};
\draw[->, red] (u) to (v);
 \node at (-1,0.25) {$\ss a$}; 

\draw[->, red] (v) to [out=45, in=135]  (w);
\node at (1,0.7){$\ss b$};
\draw[->, red] (w) to [out=-135, in=-45]  (v);
\node at (1,-0.7) {$\ss c$};

\end{tikzpicture}
\]
Define the self-similar action $(G,E)$ given by
 \[a\cdot e_1=e_2,\;\; a|_{e_1}=u,\;\; a\cdot e_3=e_6,\;\; a|_{e_3}=b,\]
\[b\cdot e_2=e_5,\;\; b|_{e_2}=a, \;\; b\cdot e_6=e_4, \;\; b|_{e_6}=c,\] 
\[c\cdot e_4=e_2,\;\; c|_{e_4}=a^{-1},\;\; c\cdot e_5=e_6,\;\; c|_{e_5}=b.\]
The actions of $a^{-1}, b^{-1}, c^{-1}$ and their restrictions are then uniquely determined:
\[a^{-1}\cdot e_2=e_1,\;\; a^{-1}|_{e_2}=u,\;\; a^{-1}\cdot e_6=e_3,\;\; a^{-1}|_{e_6}=b^{-1},\]
\[b^{-1}\cdot e_5=e_2,\;\;b^{-1}|_{e_5}=a^{-1},\;\; b^{-1}\cdot e_4=e_6,\;\; b^{-1}|_{e_4}=c^{-1},\]
\[c^{-1}\cdot e_2=e_4,\;\; c^{-1}|_{e_2}=a,\;\; c^{-1}\cdot e_6=e_5,\;\;c^{-1}|_{e_6}=b^{-1}.\]  
The actions of $u,v,w$ and their restrictions are given by
\[u\cdot e_1=e_1,\; u|_{e_1}=u, \; u\cdot e_3=e_3, \; u|_{e_3}=v, \;  v\cdot e_2=e_2, \; v|_{e_2}=u,\]
\[ v\cdot e_6=e_6, \; v|_{e_6}=w,\;  w\cdot e_4=e_4, \; w|_{e_4}=v,\;  w\cdot e_5=e_5, \; w|_{e_5}=v.\]
We can also characterize the self-similar action $(G,E)$ by the formulas
\[a\cdot e_1\mu=e_2\mu, \;  \; a\cdot e_3\mu =e_6(b\cdot\mu),\; \; b\cdot(e_2\mu)=e_5(a\cdot\mu), \]\[ b\cdot e_6\mu=e_4(c\cdot\mu),\;\; c\cdot e_4\mu=e_2(a^{-1}\cdot\mu), \;\; c\cdot e_5\mu=e_6(b\cdot\mu)\]
where $\mu\in E^*$, and these  determine uniquely an action of  $G$ on the graph $T_E$. 

Note that the action of $G$ is level transitive. It can be shown (see 
\cite{D}) that $G$ is a transitive groupoid with isotropy $\ZZ$ and therefore $C^*(G)\cong M_3 \otimes C(\TT)$. 
\end{example}

We observe that the graph $E$ itself does not have many symmetries. If a group acts on $E$ by preserving the range and source maps, then all vertices must be fixed, and the action should also fix all edges except that it could interchange $e_4$ and $e_5$, so Aut$(E)\cong \ZZ_2$. The set of self-similar groupoid actions $(G,E)$ is much richer.

\begin{example} Let $E$ be the graph
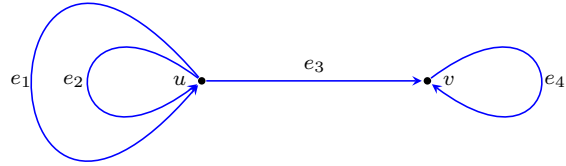
\vspace{-20mm} 
\[
\begin{tikzpicture}[shorten >=0.4pt,>=stealth, semithick]
\renewcommand{\ss}{\scriptstyle}
\node[inner sep=1.0pt, circle, fill=black]  (u) at (-1.5,0) {};
\node[left] at (u.west)  {$\ss u$};
\node[inner sep=1.0pt, circle, fill=black]  (v) at (1.5,0) {};
\node[right] at (v.east)  {$\ss v$};
\draw[->, blue] (u) to (v);
\node at (0,0.2){$\ss e_3$};
\draw[->, blue] (u) .. controls (-4.5,3.5) and (-4.5, -3.5) .. (u);
\node at (-3.9,0) {$\ss e_1$};
\draw[->, blue] (u) .. controls (-3.5,1.5) and (-3.5, -1.5) .. (u);
\node at (-3.2,0) {$\ss e_2$};
\draw[->, blue] (v) .. controls (3.5,1.5) and (3.5, -1.5) .. (v);
\node at (3.2,0) {$\ss e_4$};
\end{tikzpicture}
\vspace{-20mm}
\]
with $E^0=\{u,v\}, E^1=\{e_1, e_2, e_3, e_4\}$ and forest $T_E$
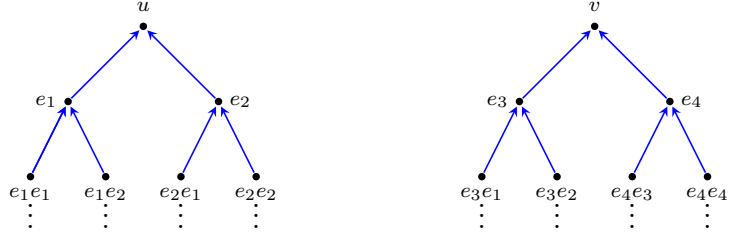
\[
\begin{tikzpicture}[shorten >=0.4pt,>=stealth, semithick]
\renewcommand{\ss}{\scriptstyle}
\node[inner sep=1.0pt, circle, fill=black]  (u) at (-3,4) {};
\node[above] at (u.north)  {$\ss u$};
\node[inner sep=1.0pt, circle, fill=black]  (v) at (3,4) {};
\node[above] at (v.north)  {$\ss v$};

\node[inner sep=1.0pt, circle, fill=black] (e1) at (-4,3){};
\node[left] at (-4,3)  {$\ss e_1$};
\node[inner sep=1.0pt, circle, fill=black] (e2) at (-2,3){};
\node[right] at (-2,3)  {$\ss e_2$};
\draw[->, blue] (e1) to  (u);
\draw[-> , blue] (e2) to  (u);
\node[inner sep=1.0pt, circle, fill=black] (e11) at (-4.5,2){};
\node[inner sep=1.0pt, circle, fill=black] (e12) at (-3.5,2){};
\node[below] at (-4.5,2)  {$\ss e_1e_1$};
\node[below] at (-4.5,2) {$\vdots$};
\draw[-> , blue] (e11) to   (e1);
\node[below] at (-3.5,2)  {$\ss e_1e_2$};
\node[below] at (-3.5,2) {$\vdots$};
\draw[-> , blue] (e11) to  (e1);
\draw[-> , blue] (e12) to  (e1);
\node[inner sep=1.0pt, circle, fill=black] (e21) at (-2.5,2){};
\node[inner sep=1.0pt, circle, fill=black] (e22) at (-1.5,2){};
\node[below] at (-1.5,2) {$\ss e_2e_2$};
\node[below] at (-1.5,2) {$\vdots$};
\node[below] at (-2.5,2) {$\ss e_2e_1$};
\node[below] at (-2.5,2) {$\vdots$};
\draw[-> , blue] (e21) to  (e2);
\draw[-> , blue] (e22) to  (e2);

\node[inner sep=1.0pt, circle, fill=black] (e3) at (2,3){};
\node[left] at (2,3)  {$\ss e_3$};

\node[inner sep=1.0pt, circle, fill=black] (e4) at (4,3){};
\node[right] at (4,3)  {$\ss e_4$};
\draw[-> , blue] (e3) to  (v);
\draw[-> , blue] (e4) to  (v);
\node[inner sep=1.0pt, circle, fill=black] (e31) at (1.5,2){};
\node[below] at (1.5,2)  {$\ss e_3e_1$};
\node[below] at (1.5,2) {$\vdots$};
\node[inner sep=1.0pt, circle, fill=black] (e32) at (2.5,2){};
\node[below] at (2.5,2)  {$\ss e_3e_2$};
\node[below] at (2.5,2) {$\vdots$};
\draw[-> , blue] (e31) to   (e3);

\draw[-> , blue] (e32) to  (e3);
\node[inner sep=1.0pt, circle, fill=black] (e43) at (3.5,2){};
\node[below] at (3.5,2)  {$\ss e_4e_3$};
\node[below] at (3.5,2) {$\vdots$};
\node[inner sep=1.0pt, circle, fill=black] (e44) at (4.5,2){};
\node[below] at (4.5,2)  {$\ss e_4e_4$};
\node[below] at (4.5,2) {$\vdots$};
\draw[-> , blue] (e43) to   (e4);
\draw[-> , blue] (e44) to   (e4);
\end{tikzpicture}
\]

Consider the groupoid $G=\la a,b,c\ra$ with $d(a)=d(b)=d(c)=u= t(a)=t(b)$ and $ t(c)=v$ as in the picture
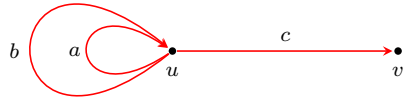
\vspace{-10mm}
\[\begin{tikzpicture}[shorten >=0.4pt,>=stealth, semithick]
\renewcommand{\ss}{\scriptstyle}
\node[inner sep=1.0pt, circle, fill=black]  (u) at (-2,0) {};
\node[below] at (u.south)  {$\ss u$};
\node[inner sep=1.0pt, circle, fill=black]  (v) at (1,0) {};
\node[below] at (v.south)  {$\ss v$};

\draw[->, red] (u) to  (v);
\node at (-0.5,0.2){$\ss c$};
\draw[->, red]  (u) .. controls (-3.5,-1) and (-3.5, 1) .. (u);
\node at (-3.3,0) {$\ss a$};
\draw[->, red] (u) .. controls (-4.5,-2) and (-4.5, 2) .. (u);
\node at (-4.1,0){$\ss b$};

\end{tikzpicture}
\vspace{-10mm}\]
and the self-similar action 

\[a\cdot e_1=e_2,\; a|_{e_1}=b, \; a\cdot e_2=e_1, \; a|_{e_2}=a,\]
\[b\cdot e_1=e_1, \; b|_{e_1}=b, \; b\cdot e_2=e_2, \; b|_{e_2}=a,\]
\[c\cdot e_1=e_3, \;  c|_{e_1}=a, \; c\cdot e_2=e_4, \; c|_{e_2}=c.\]

Then $G_u^u=\la a,b\ra$ is isomorphic to the lamplighter group $L=(\ZZ_2)^\ZZ\rtimes \ZZ$ and $C^*(G)\cong M_2\otimes C^*(L)$.

\end{example}

\begin{example}\label{ka} 

Let $E$ be the graph 

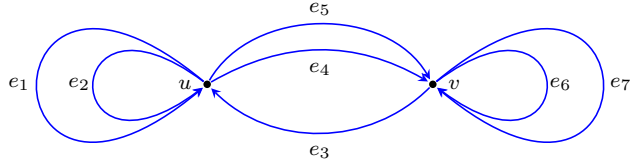
\[
\vspace{-10mm}
\begin{tikzpicture}[shorten >=0.4pt,>=stealth, semithick]
\renewcommand{\ss}{\scriptstyle}
\node[inner sep=1.0pt, circle, fill=black]  (u) at (-1.5,0) {};
\node[left] at (u.west)  {$\ss u$};

\node[inner sep=1.0pt, circle, fill=black]  (v) at (1.5,0) {};
\node[right] at (v.east)  {$\ss v$};

\draw[->, blue] (u) to [out=30, in=150]  (v);
\node at (0,0.2){$\ss e_4$};
\node at (0,1){$\ss e_5$};
\node at (0,-0.9){$\ss e_3$};
\draw[->, blue] (v) to [out=-135, in=-45]  (u);
\draw[->, blue] (u) to [out=60, in=120]  (v);
\draw[->, blue] (u) .. controls (-4.5,2.5) and (-4.5, -2.5) .. (u);
\node at (-4,0) {$\ss e_1$};
\draw[->, blue] (u) .. controls (-3.5,1.5) and (-3.5, -1.5) .. (u);
\node at (-3.2,0) {$\ss e_2$};
\draw[->, blue] (v) .. controls (3.5,1.5) and (3.5, -1.5) .. (v);
\draw[->, blue] (v) .. controls (4.5,2.5) and (4.5, -2.5) .. (v);
\node at (3.2,0) {$\ss e_6$};
\node at (4,0) {$\ss e_7$};
\end{tikzpicture}
\] 
with $E^0=\{u,v\}, E^1=\{e_1, e_2, e_3, e_4, e_5, e_6, e_7\}$ and forest $T_E$

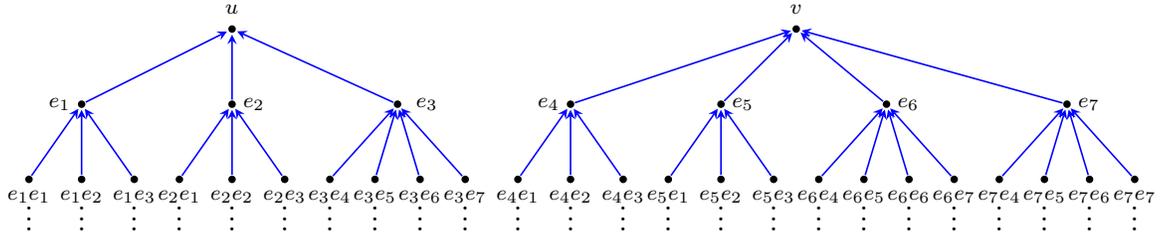
\[
\begin{tikzpicture}[shorten >=0.4pt,>=stealth, semithick]
\renewcommand{\ss}{\scriptstyle}
\node[inner sep=1.0pt, circle, fill=black]  (u) at (-3,4) {};
\node[above] at (u.north)  {$\ss u$};
\node[inner sep=1.0pt, circle, fill=black]  (v) at (4.5,4) {};
\node[above] at (v.north)  {$\ss v$};

\node[inner sep=1.0pt, circle, fill=black] (e1) at (-5,3){};
\node[left] at (-5,3)  {$\ss e_1$};
\node[inner sep=1.0pt, circle, fill=black] (e2) at (-3,3){};
\node[right] at (-3,3)  {$\ss e_2$};
\node[inner sep=1.0pt, circle, fill=black] (e3) at (-0.8,3){};
\node[right] at (-0.7,3)  {$\ss e_3$};
\draw[->, blue] (e1) to  (u);
\draw[-> , blue] (e2) to  (u);
\draw[-> , blue] (e3) to  (u);
\node[inner sep=1.0pt, circle, fill=black] (e11) at (-5.7,2){};
\node[inner sep=1.0pt, circle, fill=black] (e12) at (-5,2){};
\node[inner sep=1.0pt, circle, fill=black] (e13) at (-4.3,2){};

\node[below] at (-5.7,2)  {$\ss e_1e_1$};
\node[below] at (-5.7,2) {$\vdots$};
\draw[-> , blue] (e11) to   (e1);
\node[below] at (-5,2)  {$\ss e_1e_2$};
\node[below] at (-5,2) {$\vdots$};
\node[below] at (-4.3,2) {$\vdots$};
\node[below] at (-4.3,2)  {$\ss e_1e_3$};
\draw[-> , blue] (e13) to  (e1);
\draw[-> , blue] (e12) to  (e1);
\node[inner sep=1.0pt, circle, fill=black] (e21) at (-3.7,2){};
\node[inner sep=1.0pt, circle, fill=black] (e22) at (-3,2){};
\node[inner sep=1.0pt, circle, fill=black] (e23) at (-2.3,2){};
\node[below] at (-3.7,2) {$\ss e_2e_1$};
\node[below] at (-3.7,2) {$\vdots$};
\node[below] at (-2.3,2) {$\vdots$};
\node[below] at (-3,2) {$\ss e_2e_2$};
\node[below] at (-2.3,2) {$\ss e_2e_3$};
\node[below] at (-3,2) {$\vdots$};
\draw[-> , blue] (e21) to  (e2);
\draw[-> , blue] (e22) to  (e2);
\draw[-> , blue] (e23) to  (e2);
\node[inner sep=1.0pt, circle, fill=black] (e34) at (-1.7,2){};
\node[inner sep=1.0pt, circle, fill=black] (e35) at (-1.1,2){};
\node[inner sep=1.0pt, circle, fill=black] (e36) at (-0.5,2){};
\node[inner sep=1.0pt, circle, fill=black] (e37) at (0.1,2){};

\draw[-> , blue] (e34) to   (e3);
\draw[-> , blue] (e35) to   (e3);
\draw[-> , blue] (e36) to   (e3);
\draw[-> , blue] (e37) to   (e3);

\node[below] at (-1.7,2)  {$\ss e_3e_4$};
\node[below] at (-1.1,2)  {$\ss e_3e_5$};
\node[below] at (-0.5,2)  {$\ss e_3e_6$};
\node[below] at (0.1,2)  {$\ss e_3e_7$};

\node[below] at (-1.7,2) {$\vdots$};
\node[below] at (-1.1,2) {$\vdots$};
\node[below] at (-0.5,2) {$\vdots$};
\node[below] at (0.1,2) {$\vdots$};

\node[inner sep=1.0pt, circle, fill=black] (e4) at (1.5,3){};
\node[left] at (1.5,3)  {$\ss e_4$};

\node[inner sep=1.0pt, circle, fill=black] (e5) at (3.5,3){};
\node[right] at (3.5,3)  {$\ss e_5$};
\node[inner sep=1.0pt, circle, fill=black] (e6) at (5.7,3){};
\node[right] at (5.7,3)  {$\ss e_6$};
\node[inner sep=1.0pt, circle, fill=black] (e7) at (8.1,3){};
\node[right] at (8.1,3)  {$\ss e_7$};
\draw[-> , blue] (e5) to  (v);
\draw[-> , blue] (e6) to  (v);
\draw[-> , blue] (e7) to  (v);
\draw[-> , blue] (e4) to  (v);
\node[inner sep=1.0pt, circle, fill=black] (e41) at (0.8,2){};
\node[below] at (0.8,2)  {$\ss e_4e_1$};
\node[below] at (0.8,2) {$\vdots$};
\node[inner sep=1.0pt, circle, fill=black] (e42) at (1.5,2){};
\node[below] at (1.5,2)  {$\ss e_4e_2$};
\node[below] at (1.5,2) {$\vdots$};
\draw[-> , blue] (e41) to   (e4);

\draw[-> , blue] (e42) to  (e4);
\node[inner sep=1.0pt, circle, fill=black] (e43) at (2.2,2){};
\node[below] at (2.2,2)  {$\ss e_4e_3$};
\node[below] at (2.2,2) {$\vdots$};

\draw[-> , blue] (e43) to   (e4);

\node[inner sep=1.0pt, circle, fill=black] (e51) at (2.8,2){};
\node[below] at (2.8,2)  {$\ss e_5e_1$};
\node[below] at (2.8,2) {$\vdots$};
\node[inner sep=1.0pt, circle, fill=black] (e52) at (3.5,2){};
\node[below] at (3.5,2)  {$\ss e_5e_2$};
\node[below] at (3.5,2) {$\vdots$};
\draw[-> , blue] (e51) to   (e5);

\draw[-> , blue] (e52) to  (e5);
\node[inner sep=1.0pt, circle, fill=black] (e53) at (4.2,2){};
\node[below] at (4.2,2)  {$\ss e_5e_3$};
\node[below] at (4.2,2) {$\vdots$};
\draw[-> , blue] (e53) to   (e5);

\node[inner sep=1.0pt, circle, fill=black] (e64) at (4.8,2){};
\node[inner sep=1.0pt, circle, fill=black] (e65) at (5.4,2){};
\node[inner sep=1.0pt, circle, fill=black] (e66) at (6,2){};
\node[inner sep=1.0pt, circle, fill=black] (e67) at (6.6,2){};

\draw[-> , blue] (e64) to   (e6);
\draw[-> , blue] (e65) to   (e6);
\draw[-> , blue] (e66) to   (e6);
\draw[-> , blue] (e67) to   (e6);

\node[below] at (4.8,2)  {$\ss e_6e_4$};
\node[below] at (5.4,2)  {$\ss e_6e_5$};
\node[below] at (6,2)  {$\ss e_6e_6$};
\node[below] at (6.6,2)  {$\ss e_6e_7$};

\node[below] at (4.8,2) {$\vdots$};
\node[below] at (5.4,2) {$\vdots$};
\node[below] at (6,2) {$\vdots$};
\node[below] at (6.6,2) {$\vdots$};

\node[inner sep=1.0pt, circle, fill=black] (e74) at (7.2,2){};
\node[inner sep=1.0pt, circle, fill=black] (e75) at (7.8,2){};
\node[inner sep=1.0pt, circle, fill=black] (e76) at (8.4,2){};
\node[inner sep=1.0pt, circle, fill=black] (e77) at (9,2){};

\draw[-> , blue] (e74) to   (e7);
\draw[-> , blue] (e75) to   (e7);
\draw[-> , blue] (e76) to   (e7);
\draw[-> , blue] (e77) to   (e7);

\node[below] at (7.2,2)  {$\ss e_7e_4$};
\node[below] at (7.8,2)  {$\ss e_7e_5$};
\node[below] at (8.4,2)  {$\ss e_7e_6$};
\node[below] at (9,2)  {$\ss e_7e_7$};

\node[below] at (7.2,2) {$\vdots$};
\node[below] at (7.8,2) {$\vdots$};
\node[below] at (8.4,2) {$\vdots$};
\node[below] at (9,2) {$\vdots$};

\end{tikzpicture}
\]
The groupoid $G$ with generators $a,b$ such that $d(a)=t(a)=u,\; d(b)=t(b)=v$ acts on $E^*$ by
\[a\cdot e_1=e_2,\;\; a|_{e_1}=u,\;\; a\cdot e_2=e_1, \;\; a|_{e_2}=a, \;\; a\cdot e_3=e_3,\;\; a|_{e_3}=v,\]
\[b\cdot e_4=e_4,\;\; b|_{e_4}=a,\;\; b\cdot e_5=e_5,\;\; b|_{e_5}=a,\]
\[b\cdot e_6=e_7, \;\; b|_{e_6}=v,\;\; b\cdot e_7=e_6, \;\; b|_{e_7}=b.\]

It follows that $G$ is a group bundle with $G_u^u\cong G_v^v\cong \ZZ$, so $C^*(G)\cong C(\TT)\oplus C(\TT)$. This self-similar action is related in \cite{LRRW} and \cite{EP} to the Katsura algebra $\Oo_{A,B}$ with matrices
\[A=\left(\begin{array}{cc} 2&1\\2&2\end{array}\right),\;\; B=\left(\begin{array}{cc} 1&0\\2&1\end{array}\right).\]
\end{example}

\begin{dfn} A self similar action $(G,E)$ is called contracting if there is a finite subset $N$ of $G$ such that for every $g\in G$ there is $n\ge 0$ such that $g|_\mu\in N$ for every $\mu\in d(g)E^*$ of length $|\mu|\ge n$. The smallest such finite subset of $G$ is called the nucleus of the groupoid, denoted $\Nn$. 

The Moore diagram for $\Nn$ is the labelled directed graph with vertex set $\Nn$ and, for each $g \in \Nn$ and $e \in d(g)E^1$, an edge from $g\in \Nn$ to $g|_e \in \Nn$ labelled $(e, g\cdot e)$.  A typical edge in a Moore diagram looks like

\[\begin{tikzpicture}[shorten >=0.4pt,>=stealth, thick]
\node[inner sep=1.0pt, circle, fill=black]  (u) at (-2,0) {};
\node[inner sep=1.0pt, circle, fill=black]  (v) at (	2,0) {};
\draw[-> , black] (u) to   (v);
\node[above] at (0,0)  {$ (e,g\cdot e)$};
\node[left] at (-2,0)  {$ g$};
\node[right] at (2,0)  {$g|_e$};
\end{tikzpicture}
\]
The self-similarity relations for the set $\Nn$ can be read off the Moore diagram: the edge above encodes the relation $g\cdot (e\mu) = (g \cdot e)(g|_e \cdot\mu)$ for $\mu\in  s(e)E^*$.
\end{dfn}
\begin{example}
Consider the self-similar groupoid action $(G,E)$, where $E$ is the graph
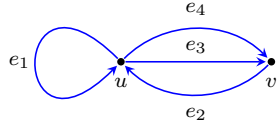
\vspace{-5mm}
\[\begin{tikzpicture}[shorten >=0.4pt,>=stealth, semithick]
\renewcommand{\ss}{\scriptstyle}
\node[inner sep=1.0pt, circle, fill=black]  (u) at (-1,0) {};
\node[below] at (u.south)  {$\ss u$};
\node[inner sep=1.0pt, circle, fill=black]  (v) at (1,0) {};
\node[below] at (v.south)  {$\ss v$};

\draw[->, blue] (u) to (v);
 \node at (0,0.7) {$\ss e_4$}; 
\draw[->, blue] (u) .. controls (-2.5,1.5) and (-2.5, -1.5) .. (u);
\node at (-2.35,0) {$\ss e_1$};
\draw[->, blue] (u) to [out=45, in=135]  (v);
\node at (0,0.2){$\ss e_3$};
\draw[->, blue] (v) to [out=-135, in=-45]  (u);
\node at (0,-0.7) {$\ss e_2$};

\end{tikzpicture}
\vspace{-10mm}
\]
and $G=\la a,b\ra$ where
\[a\cdot e_1=e_4, \; a|_{e_1}=u,\; a\cdot e_2=e_3, \; a|_{e_2}=b,\]
\[b\cdot e_3=e_1, \; b|_{e_3}=u,\; b\cdot e_4=e_2, \; b|_{e_4}=a.\]

Then $(G,E)$ contracts to $\Nn=\{u,v,a,a^{-1}, b, b^{-1}\}$ and the corresponding Moore diagram is
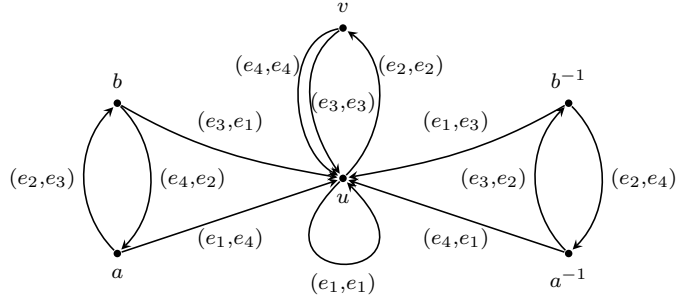
\[\begin{tikzpicture}[shorten >=0.4pt,>=stealth, semithick]
\renewcommand{\ss}{\scriptstyle}
\node[inner sep=1.0pt, circle, fill=black]  (u) at (-3,-1) {};
\node[below] at (u.south)  {$\ss a$};
\node[inner sep=1.0pt, circle, fill=black]  (v) at (-3,1) {};
\node[above] at (v.north)  {$\ss b$};
\node[inner sep=1.0pt, circle, fill=black]  (w) at (0,0) {};
\node[below] at (w.south)  {$\ss u$};
\node[inner sep=1.0pt, circle, fill=black]  (x) at (0,2) {};
\node[above] at (x.north)  {$\ss v$};

\node[inner sep=1.0pt, circle, fill=black]  (y) at (3,-1) {};
\node[below] at (y.south)  {$\ss a^{-1}$};
\node[inner sep=1.0pt, circle, fill=black]  (z) at (3,1) {};
\node[above] at (z.north)  {$\ss b^{-1}$};
\draw[->, black] (u) to [out=135, in=-135]  (v);
\node at (-4,0){$\ss (e_2,e_3)$};
\draw[->, black] (v) to [out=-45, in=45]  (u);
\node at (-2,0) {$\ss (e_4,e_2)$};
\draw[->, black] (x) to [out=190, in=150]  (w);
\node at (-1,1.5){$\ss (e_4, e_4)$};
\draw[->, black] (x) to [out=215, in=125] (w);
\node at (0,1){$\ss (e_3, e_3)$};
\draw[->, black] (w) to [out=45, in=-35]  (x);
\node at (0.9,1.5) {$\ss (e_2,e_2)$};
\draw[->, black] (v) to [out=-30, in=170] (w);
\node at (-1.5,0.75) {$\ss (e_3,e_1)$};
\draw[->, black] (u) to  (w);
\node at (-1.5,-0.8) {$\ss (e_1,e_4)$};
\draw[->, black] (y) to (w);
\node at (1.5,0.75) {$\ss (e_1,e_3)$};
\draw[->, black] (z) to [out=210, in=10] (w);
\node at (1.5,-0.8) {$\ss (e_4,e_1)$};
\draw[->, black] (w) .. controls (-1.5,-1.5) and (1.5, -1.5) .. (w);
\node at (0,-1.4) {$\ss (e_1, e_1)$};
\draw[->, black] (y) to [out=135, in=-135]  (z);
\node at (2,0){$\ss (e_3,e_2)$};
\draw[->, black] (z) to [out=-45, in=45]  (y);
\node at (4,0) {$\ss (e_2,e_4)$};
\end{tikzpicture}\]
see Proposition 9.2 in \cite{LRRW}.
\end{example}

\bigskip

\bigskip

\section{Groupoids and $C^*$-algebras associated with self-similar actions}

\bigskip

We first recall some constructions and results from \cite{D}.
\begin{dfn}
Given a self-similar groupoid action $(G,E)$, the $C^*$-algebra $C^*(G,E)$    is defined as the Cuntz-Pimsner algebra of the $C^*$-correspondence \[\Mm=\Mm(G,E)=\Xx(E)\otimes_{C(E^0)}C^*(G)\] over $C^*(G)$.   Here $\Xx(E)=C(E^1)$ is the $C^*$-correspondence over $C(E^0)$ associated to the graph $E$ and $C(E^0)=C(G^{(0)})\subseteq C^*(G)$. The right action of $C^*(G)$ on $\Mm$ is the usual one and the left action is determined by the  representation
 \[W:G\to \Ll(\Mm), \;\; W_g(i_e\otimes a)=\begin{cases} i_{g\cdot e}\otimes i_{g|_e}a\;\;\text{if}\; d(g)=r(e)\\0\;\;\text{otherwise,}\end{cases}\]
 where $ i_e\in C(E^1)$ and $i_g\in C_c(G)$ are point masses for $e\in E^1, g\in G$ and  $a\in C^*(G)$. The inner product of $\Mm$ is given by
\[\la \xi\otimes a,\eta\otimes b\ra=\la\la\eta,\xi\ra\; a,b\ra=a^*\la \xi,\eta\ra\; b\]
for $\xi, \eta\in C(E^1)$ and $a,b\in C^*(G)$. 
\end{dfn}
\begin{rmk}
 The elements $i_e\otimes 1$ for $e\in E^1$ form a Parseval frame for $\Mm$ and every $\zeta\in \Mm$ is a finite sum \[\zeta=\sum_{e\in E^1}i_e\otimes\la i_e\otimes 1,\zeta\ra.\]
In particular, if $\Xx(E)^*$ denotes the dual $C^*$-correspondence, then 
\[\Ll(\Mm)=\Kk(\Mm)\cong \Xx(E)\otimes_{C(E^0)}C^*(G)\otimes_{C(E^0)}\Xx(E)^*\cong M_n\otimes C^*(G),\] where $n=|E^1|$. The isomorphism is given by
\[i_{e_j}\otimes i_g\otimes i_{e_k}^*\mapsto e_{jk}\otimes i_g\] for $E^1=\{e_1,...,e_n\}$ and for matrix units $e_{jk}\in M_n$. There is a unital homomorphism $\Kk(\Xx(E))\to \Kk(\Mm)$ given by \[i_e\otimes i_f^*\to i_e\otimes 1\otimes i_f^*.\]


\end{rmk}
We recall the following result, see Propositions 4.4 and 4.7 in \cite{LRRW}.
\begin{thm}\label{gen}
If $U_g, P_v$ and $S_e$ are the images of $g\in G, v\in E^0=G^{(0)}$ and of $e\in E^1$ in the Cuntz-Pimsner algebra $C^*(G,E)$, then
\begin{itemize}

\item $g\mapsto U_g$ is a representation by partial isometries of $G$ with $U_v=P_v$ for $v\in E^0$;

\item $S_e$ are partial isometries with $S_e^*S_e=P_{s(e)}$ and $\ds \sum_{r(e)=v}S_eS_e^*=P_v$;

\item  $U_gS_e=\begin{cases}S_{g\cdot e}U_{g|_e}\;\mbox{if}\; d(g)=r(e)\\0,\;\mbox{otherwise;}\end{cases}$

\item $ U_gP_v=\begin{cases}P_{g\cdot v}U_g\;\mbox{if}\; d(g)=v\\0,\;\mbox{otherwise.}\;\end{cases}$
\end{itemize}
There is a gauge action $\gamma$ of $\TT$ on $C^*(G,E)$ such that $\gamma_z(U_g)=U_g,$ and $\gamma_z(S_e)=zS_e$ for $z\in \TT$.

Given $\mu=e_1\cdots e_n\in E^*$ with $e_i\in E^1$, we let $S_\mu:=S_{e_1}\cdots S_{e_n}$. Then $C^*(G,E)$ is the closed linear span of elements $S_\mu U_gS_\nu^*$, where $\mu, \nu\in E^*$ and $g\in G_{s(\nu)}^{s(\mu)}$.
\end{thm}

\begin{rmk}
The graph $C^*$-algebra $C^*(E)$ is embedded in $C^*(G,E)$ by considering elements $S_\mu S_\nu^*=S_\mu U_gS_\nu^*$ where $g=s(\nu)=s(\mu)$.
\end{rmk}

For each $k\ge 1$, consider $\Ff_k$ the closed linear span of elements $S_\mu U_gS_\nu^*$ with $\mu,\nu\in E^k$  and $g\in G_{s(\nu)}^{s(\mu)}$. Then the fixed point algebra  $\Ff(G,E):=C^*(G,E)^{\TT}$ under the gauge action is isomorphic to $\ds \varinjlim \Ff_k$. We have \[\Ff_k\cong \Ll(\Mm^{\otimes k})\cong \Xx(E)^{\otimes k}\otimes_{C(E^0)}C^*(G)\otimes_{C(E^0)}\Xx(E)^{*\otimes k}\]
using the map $S_\mu U_gS_\nu^*\mapsto i_\mu\otimes i_g\otimes i_\nu^*$, where  $i_\mu\in \Xx(E)^{\otimes k}=C(E^k)$ are point masses. The embeddings $\Ff_k\hookrightarrow \Ff_{k+1}$ are given by\[\phi_k(i_{\mu}\otimes i_g\otimes i_\nu^*)=\begin{cases}\ds \sum_{x\in d(g)E^1}i_{\mu y} \otimes i_{g|_x}\otimes i_{\nu x}^*,\;\text{if}\; g\in G_{s(\nu)}^{s(\mu)}\;\text{and}\; g\cdot x=y \\0,\;\text{ otherwise.}\end{cases}\]

\bigskip

\begin{dfn}
A self-similar groupoid action $(G,E)$ is called pseudo free if for every $g\in G$  and every $e\in d(g)E^1$, the condition $g\cdot e=e$ and $g|_{e}=s(e)$ implies that $g=r(e)$.
 \end{dfn}
 
 For example, the self-similar action in Example \ref{forest} is pseudo free.
\begin{rmk}
If $(G,E)$ is pseudo free,  then $g_1\cdot \alpha=g_2\cdot \alpha$ and $g_1|_\alpha=g_2|_\alpha$ for some $\alpha\in E^*$ implies $g_1=g_2$.
\end{rmk}

\begin{prop}  Given a self-similar groupoid action $(G,E)$,  there is an  inverse semigroup 
\[\Ss(G,E)=\{(\alpha, g, \beta): \alpha, \beta\in E^*, g\in G_{s(\beta)}^{s(\alpha)} \}\cup\{0\}\]
 with operations
\[(\alpha, g, \beta)(\gamma, h, \zeta)=\begin{cases}(\alpha, g(h|_{h^{-1}\cdot\mu}), \zeta(h^{-1}\cdot \mu))&\;\text{if}\; \beta=\gamma\mu\\(\alpha(g\cdot\mu), g|_{\mu}h, \zeta)&\;\text{if}\; \gamma=\beta\mu\\0&\;\text{otherwise}\end{cases}\]
and $(\alpha, g, \beta)^*=(\beta, g^{-1}, \alpha)$. 

\end{prop}

\begin{rmk}
The inverse semigroup $\Ss(G,E)$ acts on the infinite path space $E^\infty$ by partial homeomorphisms. The action of $(\alpha, g, \beta)\in \Ss(G,E)$ on $\xi=\beta\mu\in \beta E^\infty$ is given by \[(\alpha, g, \beta)\cdot \beta\mu=\alpha(g\cdot\mu)\in \alpha E^\infty.\]  Note that \[r(g\cdot \mu)=g\cdot r(\mu)=g\cdot s(\beta)=s(\alpha),\] so the action is well defined.
\end{rmk}
\begin{thm}
If the self-similar action $(G, E)$ is pseudo free,  then there is a locally compact  Hausdorff  \'etale groupoid $\Gg(G,E)$ such  that \[C^*(G, E)\cong C^*(\Gg(G,E)).\] 
If $G$ is amenable, then $C^*(G, E)$ is nuclear and $\Gg(G,E)$ is also amenable.
\end{thm}
\begin{proof}
Consider the groupoid of germs associated with $(\Ss(G,E), E^\infty)$:
\[\Gg(G,E)=\{[\alpha, g, \beta; \xi]: \alpha, \beta\in E^*,\; g\in G^{s(\alpha)}_{s(\beta)},\; \xi\in \beta E^\infty\}.\]
Two germs $[\alpha, g,\beta;\xi], [\alpha',g',\beta';\xi']$ in $\Gg(G,E)$ are equal if and only if $\xi=\xi'$ and there is an idempotent $z\in \Ss(G,E)\setminus \{0\}$ such that $z\cdot \xi=\xi$ and $(\alpha, g, \beta)z=(\alpha',g',\beta')z$. 
Using the form of idempotents, we obtain that $\xi=\beta\gamma\eta$ for $\gamma\in E^*$ and $\eta\in E^\infty$, with $r(\gamma)=s(\beta)$ and $r(\eta)=s(\gamma)$. Moreover,  \[\alpha'=\alpha(g\cdot \gamma), \beta'=\beta\gamma,\;\mbox{and}\; g'=g|_\gamma,\]
so
\[[\alpha, g,\beta;\beta\gamma\eta]= [\alpha(g\cdot \gamma),g|_\gamma,\beta\gamma;\beta\gamma\eta].\]
The unit space of $\Gg(G,E)$ is
\[\Gg(G,E)^{(0)}=\{[\alpha, s(\alpha), \alpha; \xi]: \xi\in \alpha E^\infty\},\] identified with $E^\infty$ by the map $[\alpha, s(\alpha), \alpha; \xi]\mapsto \xi$. 

The target and domain maps of the groupoid $\Gg(G,E)$ are given by
\[t([\alpha, g,\beta; \beta\mu])=\alpha(g\cdot\mu),\;\; d([\alpha, g, \beta;\beta\mu])=\beta\mu.\]
If two elements $\gamma_1,\gamma_2\in \Gg(G,E)$ are composable, then \[\gamma_1=[\alpha_1, g_1, \alpha_2; \alpha_2(g_2\cdot\xi)],\;\; \gamma_2=[\alpha_2, g_2, \beta; \beta\xi]\] for some $\alpha_1, \alpha_2, \beta\in E^*, \xi\in E^\infty, g_1, g_2\in G$ and in this case
\[\gamma_1\gamma_2=[\alpha_1, g_1g_2, \beta;\beta\xi].\]
In particular, $[\alpha, g,\beta;\beta\mu]^{-1}=[\beta, g^{-1}, \alpha;\alpha(g\cdot \mu)]$.

The topology on $\Gg(G,E)$ is generated by the compact open bisections of the form
\[Z(\alpha, g, \beta;U)=\{[\alpha, g, \beta; \xi]\in \Gg(G,E): \xi\in U\},\]
where $U$ is an open compact subset of $Z(\beta)=\beta E^\infty$. Since any open set $U\subseteq Z(\beta)$ is a disjoint union $\ds U=\bigsqcup_i Z(\beta\gamma_i)$ where $\gamma_i\in E^*$ with $r(\gamma_i)=s(\beta)$, we get
\[Z(\alpha, g,\beta;U)=\bigsqcup_iZ(\alpha, g, \beta;Z(\beta\gamma_i))=\bigsqcup_iZ(\alpha(g\cdot \gamma_i), g|_{\gamma_i}, \beta\gamma_i;Z(\beta\gamma_i)).\]
It follows that the sets $Z(\alpha, g, \beta)=Z(\alpha,g,\beta;Z(\beta))$ form a basis for the topology of $\Gg(G,E)$. In particular,
\[d(Z(\alpha, g, \beta))=Z(\beta)\;\;\text{and}\; \; t(Z(\alpha, g, \beta))=Z(\alpha).\]

Since $(G,E)$ is pseudo free, it follows that $[\alpha,g,\beta;\xi]=[\alpha, g',\beta;\xi]$ if and only if $g=g'$. Moreover, the groupoid $\Gg(G,E)$ is Hausdorff, see Proposition 12.1 in \cite{EP}. Using the properties given in Theorem \ref{gen} and the groupoid multiplication, the isomorphism $\phi: C^*(G,E)\to C^*(\Gg(G,E))$ is given by
 \[\phi(P_v)=\chi_{Z(v, v,v)},\]\[ \phi(T_e)=\chi_{Z(e, s(e),s(e))},\]\[\phi(U_g)=\chi_{Z(t(g),g,d(g))}\]
 for $v\in E^0, e\in E^1$ and $g\in G$. Here $\chi_Z$ is the indicator function of $Z$.

Since $G$ is amenable, it follows that $C^*(G)$ is nuclear and hence the Cuntz-Pimsner algebra $C^*(G,E)$ is nuclear. The same kind of argument as in \cite{EP} shows that $\Gg(G,E)$ is also amenable.
\end{proof}
\begin{rmk}
Note that the graph groupoid $\Gg_E$ is isomorphic to an open subgroupoid of $\Gg(G,E)$ with the same unit space $E^\infty$. It is obtained as
\[\Gg_E=\{(\alpha\xi, |\alpha|-|\beta|,\beta\xi)\in E^\infty\times \ZZ\times E^\infty: \alpha, \beta\in E^*, \; s(\alpha)=s(\beta)\}\cong\]\[\cong\{[\alpha, s(\alpha), \beta; \beta\xi]\in \Gg(G,E): \alpha, \beta\in E^*,\; s(\alpha)=s(\beta),\; \xi\in E^\infty\}.\]
\end{rmk}
\begin{dfn}
Given a self-similar action $(G,E)$, we say that $E$ is $G$-transitive if given two vertices $u,v\in E^0$, there are vertices $u_0=u,u_1,...,u_{2n}=v$ such that we can connect each $u_{2k-2}$ to $u_{2k-1}$ by a path in $E$  and there are $g_k\in G$ with $d(g_k)=u_{2k-1}$ and $t(g_k)=u_{2k}$ for all $k=1,...,n$.
\end{dfn}
Using the same method as in Theorem 13.6 and Corollary 13.7 in \cite{EP}, it follows that
\begin{thm} Given a self-similar action $(G,E)$, the groupoid $\Gg(G,E)$ is minimal iff $E$ is $G$-transitive. In particular, $\Gg(G,E)$ is minimal   if $G$ is transitive or if the adjacency matrix $A_E$ of $E$ is irreducible. The groupoid $\Gg(G,E)$ is also minimal  if the action of $G$ is level transitive, see Definition \ref{trans}.
\end{thm}

Recall that a circuit (also called a cycle or a loop) in $E$ is a  path $\alpha\in E^*$ with $|\alpha |\ge 1$ such that $s(\alpha)=r(\alpha)$. A $G$-circuit in $E$ is a pair $(g, \alpha)$ with $g\in G$ and $\alpha\in d(g)E^*$ with $|\alpha|\ge 1$ such that $s(\alpha)=g\cdot r(\alpha)$.

Note that a  circuit $\alpha$ in $E$ may be concatenated to produce an infinite path $\xi=\alpha\alpha\alpha\cdots$ in $E^\infty$. If $z=(\alpha, s(\alpha), s(\alpha))\in \Ss(G,E)$, then $z\cdot\xi=\xi$ since $r(\alpha)=s(\alpha)$. 

To create fixed points from $G$-circuits $(g,\alpha)$,
define a sequence $\{\alpha^n\}_{n\ge 1}$ of finite paths and a sequence $\{g_n\}_{n\ge 1}$ of groupoid elements recursively by $\alpha^1=\alpha, g_1=g$ and $\alpha^{n+1}=g_n\cdot\alpha^n$ where $g_{n+1}=g_n|_{\alpha^n}$ for $n\ge 1$. Then 
\[r(\alpha^2)=r(g\cdot \alpha)=g\cdot r(\alpha)=s(\alpha)\]
and assuming $r(\alpha^k)=s(\alpha^{k-1})$ for some $k\ge 2$, we get
\[r(\alpha^{k+1})=r(g_k\cdot\alpha^k)=g_k\cdot r(\alpha^k)=g_k\cdot s(\alpha^{k-1})=\]\[=g_{k-1}|_{\alpha^{k-1}}\cdot s(\alpha^{k-1})=s(g_{k-1}\cdot \alpha^{k-1})=s(\alpha^k).\]
The concatenation
$\xi=\alpha^1\alpha^2\alpha^3\cdots$ is an infinite path such that $g\cdot \xi=\alpha^2\alpha^3\alpha^4\cdots$. For any path $\beta\in E^*$ with $s(\beta)=r(\alpha)$, the infinite path $\beta\xi$ is fixed by $(\beta\alpha, g, \beta)\in \Ss(G,E)$ because
\[(\beta\alpha,g,\beta)\cdot \beta\xi=\beta\alpha(g\cdot\xi)=\beta\xi.\]

\begin{rmk}
Given $z=(\alpha, g, \beta)\in \Ss(G,E)$ with $|\alpha|>|\beta|$, then $z$ admits at most one fixed point in $E^\infty$ and assuming $z\cdot\xi=\xi$, then there is a $G$-circuit $(g,\gamma)$ such that $\alpha=\beta\gamma$ and $\xi=\beta\zeta$ with $\zeta$ constructed from $(g, \gamma)$.
\end{rmk}
Indeed, if $z\cdot\xi=\xi$, then $\xi\in Z(\beta)$ so $\xi=\beta\zeta$. Then
\[\beta\zeta=\xi=(\alpha, g, \beta)(\beta\zeta)=\alpha (g\cdot \zeta)\]
and since $|\alpha|>|\beta|$, we must have $\alpha=\beta\gamma$ such that $g\cdot r(\gamma)=g\cdot s(\beta)=s(\alpha)=s(\gamma)$, i.e. $(g,\gamma)$ is a G-circuit. Moreover, $\beta\zeta=\beta\gamma (g\cdot \zeta)$, so $\zeta=\gamma (g\cdot \zeta)$. If we write $\zeta=\gamma^1\gamma^2\gamma^3\cdots$ with $|\gamma^i|=|\gamma|$ for $i\ge 1$, then
\[\gamma^1\gamma^2\gamma^3\cdots=\gamma (g\cdot\zeta)=\gamma(g\cdot\gamma^1)(g|_{\gamma^1}\cdot\gamma^2)(g|_{\gamma^1}|_{\gamma^2}\cdot\gamma^3)\cdots.\]
It follows that $\gamma^1=\gamma$ and $\gamma^{n+1}=g_n\cdot \gamma^n$ where $g_1=g$ and $g_{n+1}=g_n|_{\gamma^n}$ for $n\ge 1$.

Using the methods of Theorem 14.10 and Corollary 14.13 in \cite{EP} to prove that the interior of the isotropy of $\Gg(G,E)$ is $E^\infty$, we obtain

\begin{thm} Suppose that we have a pseudo free self-similar action $(G,E)$. Then the  groupoid $\Gg(G,E)$ is effective if and only if

(a) every $G$-circuit has an entry (which is the same as saying that every circuit has an entry, since $E$ is finite);

(b) for every $g\in G\setminus G^{(0)}$ there is $\zeta\in Z(d(g))$ such that $g\cdot \zeta\neq \zeta$.

\end{thm}

\begin{cor} If $\Gg(G,E)$ is Hausdorff, effective and minimal, then $\Gg(G,E)$ is purely infinite.
\end{cor}

\begin{proof} This follows since the open ample subgroupoid $\Gg_E$ of $\Gg(G,E)$ is purely infinite.
\end{proof}
For example, the groupoid $\Gg(G,E)$ obtained from Example \ref{forest} is minimal and effective. In particular, its $C^*$-algebra is simple and purely infinite.

\bigskip

\bigskip

\section{Tables and the Higman-Thompson groups}

\bigskip

Consider a finite directed graph $E$ with no sources such that $E^\infty$ is a Cantor set. We recall first some facts about the Higman-Thompson group of $E$.

\begin{dfn} A table over the graph $E$ is a matrix of the form
\[\tau=\left(\begin{array}{cccc} \alpha_1&\alpha_2&\cdots&\alpha_m\\\beta_1&\beta_2&\cdots&\beta_m\end{array}\right),\]
where $\alpha_i, \beta_i\in E^*$ are such that   $s(\alpha_i)=s(\beta_i)$ and  $E^\infty$ is decomposed into disjoint unions
\[E^\infty=\bigsqcup_{i=1}^mZ(\alpha_i)=\bigsqcup_{j=1}^mZ(\beta_j),\]
i.e. for every infinite path $\mu\in E^\infty$ exactly one $\alpha_i$ and exactly one $\beta_j$ is a prefix of $\mu$. 
\end{dfn}

It follows that the finite paths in each row of a table are incomparable (none of them is a begining of the other) and the corresponding cylinder sets $Z(\alpha_i)$ and $Z(\beta_j)$ are indeed disjoint. Every table $\tau$ defines a homeomorphism $\bar{\tau}$ of $E^\infty$ by the rule 
\[\bar{\tau}(\beta_i\mu)=\alpha_i\mu, \;\text{ for }\; i=1,...,m.\]
This homeomorphism does not change if we permute the columns of $\tau$. Its inverse is given by
\[\tau^{-1}=\left(\begin{array}{cccc}\beta_1&\beta_2&\cdots&\beta_m\\ \alpha_1&\alpha_2&\cdots&\alpha_m\end{array}\right).\]
A split of a table is a matrix obtained from the original one by several replacements of a column $\left(\begin{array}{c}\alpha_i\\\beta_i\end{array}\right)$ by a matrix
\[\left(\begin{array}{cccc}\alpha_ie_1&\alpha_ie_2&\cdots&\alpha_ie_d\\ \beta_ie_1&\beta_ie_2&\cdots&\beta_ie_d\end{array}\right),\]
where $s(\alpha_i)E^1=\{e_1, e_2,...,e_d\}=s(\beta_i)E^1$. It is easy to check that a split of a table is also a table, and two tables $\tau_1,\tau_2$ define the same homeomorphisms of $E^\infty$ if they have splits which agree up to a permutation of the columns.

The set of all homeomorphisms of $E^\infty$ defined by such tables is a subgroup of Homeo$(E^\infty)$, called the Higman-Thompson group of $E$ and denoted by $V_E$. In \cite{MM}, the authors define tables associated to a one-sided topological Markov shift $(X_A,\sigma_A)$ for an irreducible $N\times N$ matrix $A$ with entries in $\{0, 1\}$. They studied the corresponding Higman-Thompson group and proved that it is isomorphic to the topological full group of the  groupoid associated to $(X_A,\sigma_A)$. Their results were generalized for graph groupoids in \cite{NO}. 

\begin{rmk}
The correspondence \[\tau=\left(\begin{array}{cccc} \alpha_1&\alpha_2&\cdots&\alpha_m\\\beta_1&\beta_2&\cdots&\beta_m\end{array}\right)\mapsto T=S_{\alpha_1}S_{\beta_1}^*+S_{\alpha_2}S_{\beta_2}^*+\cdots +S_{\alpha_m}S_{\beta_m}^*\] defines a faithful unitary representation of the group $V_E$ in the $C^*$-algebra of the graph $E$. Moreover, \[T\in N(C(E^\infty), C^*(E))=\{u\in U(C^*(E)): uC(E^\infty)u^*=C(E^\infty)\}\] and $f\circ \bar{\tau}^{-1}=T fT^*$ for $f\in C(E^\infty)$.
\end{rmk}

\begin{proof}

Indeed, as in \cite{MM}, $T=S_{\alpha_1}S_{\beta_1}^*+S_{\alpha_2}S_{\beta_2}^*+\cdots +S_{\alpha_m}S_{\beta_m}^*$ is a unitary in $C^*(E)$ since $\ds E^\infty=\bigsqcup_{i=1}^mZ(\alpha_i)=\bigsqcup_{i=1}^mZ(\beta_i)$ and therefore \[T^*T=(\sum_{i=1}^mS_{\beta_i}S_{\alpha_i}^*)(\sum_{j=1}^mS_{\alpha_j}S_{\beta_j}^*)=\sum_{i=1}^mS_{\beta_i}S_{\beta_i}^*=\sum_{v\in E^0}P_v=I,\]
\[TT^*=(\sum_{i=1}^mS_{\alpha_i}S_{\beta_i}^*)(\sum_{j=1}^mS_{\beta_j}S_{\alpha_j}^*)=\sum_{i=1}^mS_{\alpha_i}S_{\alpha_i}^*=\sum_{v\in E^0}P_v=I.\]
For $f=\chi_{Z(\eta)}=S_\eta S_\eta^*\in C(E^\infty)$ a computation shows that $T\chi_{Z(\eta)}T^*=\chi_{Z(\eta)}\circ \bar{\tau}^{-1}$.
\end{proof}
\begin{example}
For the graph $E$ in Example \ref{ex}, a table is 
\[\tau=\left(\begin{array}{cccc}u&e_4&e_5&v\\u&e_5&e_4&v\end{array}\right),\]
with split
\[\tau'=\left(\begin{array}{ccccc}u&e_4e_2&e_4e_6&e_5&v\\u&e_5e_2&e_5e_6&e_4&v\end{array}\right).\]
\end{example}
Consider now   a self-similar groupoid action $(G, E)$.  Assume that $|uE^1|=d$ is constant for all $u\in E^0$.
\begin{dfn}
A $G$-table is a matrix of the form
\[\tau=\left(\begin{array}{cccc}  \alpha_1&\alpha_2&\cdots&\alpha_m\\g_1&g_2&\cdots &g_m\\\beta_1&\beta_2&\cdots&\beta_m\end{array}\right),\]
where $g_i\in G_{s(\beta_i)}^{s(\alpha_i)}$  and $E^\infty=\bigsqcup_{i=1}^mZ(\alpha_i)=\bigsqcup_{i=1}^mZ(\beta_i)$.

A $G$-table $\tau$ determines a homeomorphism $\bar{\tau}$ of $E^\infty$ taking $\beta_i\mu$ into $\alpha_i(g_i\cdot \mu)$. This homeomorphism does not change if we permute the columns of $\tau$. The set of all homeomorphisms defined by such tables is a subgroup of Homeo$(E^\infty)$, called the Higman-Thompson group of $(G,E)$ and denoted by $V_E(G)$. 

\end{dfn}
The inverse of $\bar{\tau}$ corresponds to the table
\[\tau^{-1}=\left(\begin{array}{cccc} \beta_1&\beta_2&\cdots&\beta_m \\g_1^{-1}&g_2^{-1}&\cdots &g_m^{-1}\\\alpha_1&\alpha_2&\cdots&\alpha_m\end{array}\right).\]

The splitting rule for $G$-tables will replace a column $\ds\left(\begin{array}{c}\alpha\\g\\\beta\end{array}\right)$ by the matrix
\[\left(\begin{array}{cccc} \alpha f_1&\alpha f_2&\cdots&\alpha f_d\\h_1&h_2&\cdots&h_d\\\beta e_1&\beta e_2&\cdots&\beta e_d\end{array}\right),\]
where  $s(\alpha)E^1=\{f_1,f_2,...,f_d\}, s(\beta)E^1=\{e_1,e_2,...,e_d\}$, $g\cdot e_j=f_j$ and $ h_j=g|_{e_j}$ for all $j=1,...,d$. It is easy to check that a split of a table is also a table, and two tables $\tau_1,\tau_2$ define the same homeomorphisms of $E^\infty$ if they have splits which agree up to a permutation of the columns.

\begin{lem}
Given a $G$-table $\ds \tau=\left(\begin{array}{cccc}  \alpha_1&\alpha_2&\cdots&\alpha_m\\g_1&g_2&\cdots &g_m\\\beta_1&\beta_2&\cdots&\beta_m\end{array}\right)$, the correspondence \[\tau\mapsto T=S_{\alpha_1}U_{g_1}S_{\beta_1}^*+S_{\alpha_2}U_{g_2}S_{\beta_2}^*+\cdots +S_{\alpha_m}U_{g_m}S_{\beta_m}^*\] defines a faithful unitary representation of the group $V_E(G)$ in the $C^*$-algebra $C^*(G,E)$.

\end{lem}

\begin{proof}
This follows from the  computations 
\[(\sum_{i=1}^mS_{\alpha_i}U_{g_i}S_{\beta_i}^*)(\sum_{j=1}^mS_{\beta_j}U^{-1}_{g_j}S_{\alpha_i}^*)=\sum_{i=1}^mS_{\alpha_i}S_{\alpha_i}^*=\sum_{v\in E^0}P_v=I,\]
\[(\sum_{i=1}^mS_{\beta_i}U_{g_i}^{-1}S_{\alpha_i}^*)(\sum_{j=1}^mS_{\alpha_j}U_{g_j}S_{\beta_j}^*)=\sum_{i=1}^mS_{\beta_i}S_{\beta_i}^*=\sum_{v\in E^0}P_v=I.\]

\end{proof}
\begin{rmk}
The group $V_E$ can be identified with a subgroup of $V_E(G)$ such that
\[\left(\begin{array}{cccc} \alpha_1&\alpha_2&\cdots&\alpha_m\\\beta_1&\beta_2&\cdots&\beta_m\end{array}\right)\mapsto \left(\begin{array}{cccc} \alpha_1&\alpha_2&\cdots&\alpha_m\\s(\alpha_1)&s(\alpha_2)&\cdots &s(\alpha_m)\\\beta_1&\beta_2&\cdots&\beta_m\end{array}\right).\]
\end{rmk}
\begin{example}

For the groupoid action in Example \ref{forest}, a $G$-table is
\[\tau=\left(\begin{array}{cccc} u & v& e_4&e_5\\a^{-1}&cb&cb&a\\e_4&e_5&v&u\end{array}\right)\]
with split
\[\tau'=\left(\begin{array}{ccccc} e_1&e_3 & v& e_4&e_5\\u &b^{-1}&cb&cb&a\\e_4e_2&e_4e_6&e_5&v&u\end{array}\right).\]

\end{example}

\bigskip

\bigskip

\section{Topological full groups and the AH conjecture}
\bigskip

The study of full groups in the setting of topological dynamics was initiated by T. Giordano, I. F. Putnam and C. F. Skau in \cite{GPS}. For a minimal action $\varphi:\ZZ\curvearrowright X$ on a Cantor set, they defined several types of full groups and showed that these groups completely determine the orbit equivalence class, the strong orbit equivalence class and the flip conjugacy class of $\varphi$, respectively. Topological full groups associated to dynamical systems and  to \'etale groupoids are  complete invariants for continuous orbit equivalence and groupoid isomorphism.
They also provide means of constructing new groups with interesting properties, most notably by providing the first examples of finitely generated infinite simple groups that are amenable.
\begin{example}
For the  $AF$-equivalence relation associated to a Bratteli diagram $(V,E)$ with levels $(V_m, E_m)$, the topological full group is an increasing union of groups of the form $\ds \bigoplus_{v\in V_m}S_{n(v)}$, where  $n(v)$ denotes the number of paths from the top vertex $v_0$ to $v\in V_m$ and $S_n$ is the symmetric group, see Proposition 3.3 in \cite{M06}.
\end{example}
\begin{dfn} A second countable groupoid $G$ is ample if it is \'etale and $G^{(0)}$ is zero-dimensional; equivalently, $G$ is ample if it has a basis of compact open bisections. An ample  groupoid $G$ is elementary if it is compact and principal. $G$  is an $AF$ groupoid if there exists an ascending chain of open elementary subgroupoids $G_1\subseteq  G_2\subseteq ...\subseteq G$ such that $G = \bigcup_{i=1}^\infty G_i$.

\end{dfn}

We recall now the definion of homology of \'etale groupoids which was introduced by Crainic and Moerdijk in \cite{CM}. Let $A$ be an abelian group and
let $\pi: X \to Y$ be a local homeomorphism between two locally compact Hausdorff spaces. Given any $f \in C_c(X, A)$ we define
\[\pi_*(f)(y):= \sum_{\pi(x)=y} f(x).\]

It follows that $\pi_*(f)\in C_c(Y,A)$.
Given an \' etale groupoid $G$, let $G^{(1)}=G$ and for $n\ge 2$ let $G^{(n)}$ be the space of composable strings
of $n$ elements in $G$ with the product topology. For $n\ge 2$ and  $ i = 0,...,n$, we let $\partial_i : G^{(n)} \to G^{(n-1)}$ be the face maps defined by

\[\partial_i(g_1,g_2,...,g_n)=\begin{cases}(g_2,g_3,...,g_n)&\;\text{if}\; i=0,\\
 (g_1,...,g_ig_{i+1},...,g_n) &\;\text{if} \;1\le i\le n-1,\\
(g_1,g_2,...,g_{n-1})& \;\text{if}\; i = n.\end{cases},\]
which are local homeomorphisms.
We define the homomorphisms $\delta_n : C_c(G^{(n)}, A) \to C_c(G^{(n-1)}, A)$ given by 
\[\delta_1=d_*-t_*,\;\; \delta_n=\sum_{i=0}^n(-1)^i\partial_{i*}\; \text{for}\; n\ge 2.\]
Recall that $d,t:G\to G^{(0)}$ are the domain and target maps.
It can be verified that $\delta_n\circ\delta_{n+1}=0$ for all $n\ge 1$.

The homology groups $H_n(G, A)$ are by definition  the homology groups of the chain complex $C_c(G^{(*)},A)$ given by
\[
0\stackrel{\delta_0}{\longleftarrow} C_c(G^{(0)},A)\stackrel{\delta_1}{\longleftarrow}C_c(G^{(1)},A)\stackrel{\delta_2}{\longleftarrow}C_c(G^{(2)},A)\longleftarrow\cdots,\]
i.e. $H_n(G,A)=\ker \delta_n/\text{im } \delta_{n+1}$, where $\delta_0=0$.
If $A=\ZZ$, we write $H_n(G)$ for $H_n(G,\ZZ)$.
\begin{example}
For the transformation groupoid $\Gamma \ltimes X$ associated to a countable discrete group action $\Gamma \curvearrowright X$ on a Cantor set, it follows that \[H_n(\Gamma\ltimes X)\cong H_n(\Gamma, C(X,\ZZ)),\] the group homology of $\Gamma$ with coefficients in $C(X,\ZZ)$.
\end{example}
If $U\subseteq G$  is  a bisection of an \'etale groupoid $G$, then  $t|_U:  U \to t(U)$ is a homeomorphism, and similarly for $d$. Thus we get a homeomorphism $\pi_U := t|_U \circ (d|_U)^{-1}$ from $d(U)$ to $t(U)$ which maps $d(g)$ to $t(g)$ for each $g\in  U$. We say that the bisection $U$ is full if $t(U) = d(U) = G^{(0)}$, and in this case $\pi_U$ is a homeomorphism of $G^{(0)}$. For a homeomorphism $\alpha: X \to X$ of a topological space $X$ we define the support of $\alpha$ to be the set 
\[\text{supp}(\alpha) :=\overline{ \{x\in X \mid  \alpha(x)\neq x\}}.\]
\begin{dfn}
The topological full group of an effective \'etale groupoid $G$ is
\[\ldbrack G\rdbrack:= \{\pi_U \mid  U \subseteq G\;\text {full bisection} \},\]
which is a subgroup of the homeomorphism group of $G^{(0)}$ such that $\pi_U\circ \pi_V=\pi_{UV}$. The commutator subgroup of $\ldbrack G\rdbrack$ is denoted by $D(\ldbrack G\rdbrack)$ and the abelianization of $\ldbrack G\rdbrack$ is $\ldbrack G\rdbrack_{ab}=\ldbrack G\rdbrack/D(\ldbrack G\rdbrack)$.
\end{dfn}

 We observe that when $G$ is effective and Hausdorff, then $\text{supp}(\pi_U)$ is also open for any full bisection $U$. 
\begin{rmk}
For a second countable  effective ample groupoid $G$, the group $\ldbrack G\rdbrack$ is at most countable and there is a natural short exact sequence
\[1\to U(C(G^{(0)}))\to N(C(G^{(0)}),C^*_r(G))\to \ldbrack G\rdbrack\to 1,\]
where $N(C(G^{(0)}),C^*_r(G))$ denotes the group of unitary normalizers of $C(G^{(0)})$ in $C^*_r(G)$ (Proposition 5.6 in \cite{M12}).
\end{rmk}

Given a bisection $U\subseteq G$ with $t(U)\cap d(U)=\emptyset$, construct
\[\hat{U}:=U\sqcup U^{-1}\sqcup (G^{(0)}\setminus (t(U)\cup d(U))),\]
which is a full bisection. The associated homeomorphism $\pi_{\hat{U}}\in \ldbrack G\rdbrack$ satisfies 
\[\pi_{\hat{U}}(d(U))=t(U),\; \pi_{\hat{U}}(t(U))=d(U),\]\[\text{supp}(\pi_{\hat{U}})=t(U)\cup d(U),\;\pi_{\hat{U}}\circ \pi_{\hat{U}}=id_{G^{(0)}}.\]
An element $\pi\in \ldbrack G\rdbrack$ that satisfies $\pi^2=id_{G^{(0)}}$ is called a transposition.

For an effective minimal second countable Hausdorff ample groupoid $G$ with $G^{(0)}$ a Cantor set, the AH-conjecture of Matui claims that the following sequence is exact \[H_0(G)\otimes \ZZ_2\stackrel{j}{\longrightarrow } \ldbrack G\rdbrack_{ab}\stackrel{I_{ab}}{\longrightarrow }H_1(G)\to 0.\]
In particular, if the group $H_0(G)$ is $2$-divisible, then this implies that $\ldbrack G\rdbrack_{ab}\cong H_1(G)$.
The  map $I_{ab}$ is induced by the index map $I:\ldbrack G\rdbrack\to H_1(G)$ given by $\pi_U\mapsto [\chi_U]$, where $U$ is a full bisection of $G$ and $\chi_U$ is the indicator function of $U$, which belongs to $\ker \delta_1=\ker(d_*-t_*)$. The map $j$ takes $[\chi_{d(U)}]\otimes \bar{1}$ into $[\pi_{\hat{U}}]$, where $U\subseteq G$ is a bisection with $d(U)\cap t(U)=\emptyset$ and $\pi_{\hat{U}}\in \ldbrack G\rdbrack$ is the transposition defined above.

\begin{rmk}
The AH-conjecture has been verified for several cases of groupoids, but so far no counter-example has been found. It is shown by H. Matui in \cite{M15} that if the groupoid $G$ is  purely infinite, then the index map $I:\ldbrack G\rdbrack\to H_1(G)$ is surjective. Also, in this case, $\ldbrack G\rdbrack$ contains a copy of $\ZZ_2\ast \ZZ_3$, so $\ldbrack G\rdbrack$ is not amenable.
\end{rmk}

Denote by $\Tt(G)$ the subgroup of $\ldbrack G\rdbrack$ generated by all transpositions. The subgroup $\Tt(G)$  is the analog of the symmetric group. One has $\Tt(G)\subseteq \ker(I)$. Indeed, using Lemma 7.3 in \cite{M12} we have
\[I(\pi_{\hat{U}})=[\chi_{\hat{U}}]=[\chi_{U\sqcup U^{-1}\sqcup (G^{(0)}\setminus \text{supp}(\pi_{\hat{U}})}]=[\chi_U+\chi_{U^{-1}}]=0\in H_1(G).\]
 We say that $G$ has property TR if $\Tt(G)=\ker(I)$. In order to verify the AH conjecture for $G$, it suffices to establish property TR. In general, property TR implies the inclusion $\ker(I_{ab})\subseteq \text{im} (j)$. The converse holds if the commutator subgroup $D(\ldbrack G\rdbrack)$ is simple, since in this case $D(\ldbrack G\rdbrack)=\Aa(G)$, the analog of the alternating subgroup of $\Tt(G)$ defined in \cite{N19}. The group $D(\ldbrack G\rdbrack)$ is known to be simple for minimal groupoids which are either almost finite or purely infinite, see \cite{M15}.
\bigskip

We now specialize to the groupoid $\Gg(G,E)$ associated to a self-similar action $(G,E)$ such that $E^\infty$ is a Cantor set. Recall that $E$ is a finite graph with no sources  and $G^{(0)}=E^0$.
\begin{thm}
For a self-similar action $(G,E)$ such that $\Gg(G,E)$ is effective, we have $V_E(G)\cong \ldbrack \Gg(G,E)\rdbrack$. In particular, $\ldbrack \Gg_E\rdbrack\subseteq \ldbrack \Gg(G,E)\rdbrack$.
\end{thm}

\begin{proof}
Indeed, given a unitary element $\ds \sum_{i=1}^mS_{\alpha_i}U_{g_i}S_{\beta_i}^*$ of $V_E(G)$ corresponding to the $G$-table $\ds \tau=\left(\begin{array}{cccc}  \alpha_1&\alpha_2&\cdots&\alpha_m\\g_1&g_2&\cdots &g_m\\\beta_1&\beta_2&\cdots&\beta_m\end{array}\right)$, this determines a full bisection $U=\bigsqcup_{i=1}^mZ(\alpha_i, g_i, \beta_i)$ of $\Gg(G,E)$ and an element $\pi_U=\bar{\tau}$ of $\ldbrack\Gg(G,E)\rdbrack$, so we get a map $V_E(G)\to \ldbrack\Gg(G,E)\rdbrack$. It is easy to see that this map is a group homomorphism. It is injective since if $\ds \sum_{i=1}^mS_{\alpha_i}U_{g_i}S_{\beta_i}^*$ determines the identity of $\Gg(G,E)^{(0)}=E^\infty$, then $\alpha_i=\beta_i$ and  $g_i=s(\alpha_i)$ for $1\le i\le m$, so \[ \sum_{i=1}^mS_{\alpha_i}U_{g_i}S_{\beta_i}^*=\sum _{v\in E^0}P_v=I.\] To show surjectivity, given $\pi\in\ldbrack\Gg(G,E)\rdbrack$, since $\Gg(G,E)$ is effective, there is a unique full bisection $U=\bigsqcup_{i=1}^mZ(\alpha_i,g_i,\beta_i)$ with $\pi=\pi_U$ which determines an element $\ds \sum_{i=1}^mS_{\alpha_i}U_{g_i}S_{\beta_i}^*$  of $V_E(G)$.
\end{proof}

\begin{cor} Assuming that $\Gg(G,E)$ is amenable,
we have an exact sequence
\[1\to U(C(E^\infty))\to N(C(E^\infty), C^*(G,E))\to  \ldbrack \Gg(G,E)\rdbrack\to 1\]
which splits,
where \[N(C(E^\infty), C^*(G,E))=\{u\in U(C^*(G,E))\;\mid \; uC(E^\infty)u^*=C(E^\infty)\}.\]
The splitting is given by $U\mapsto \chi_U$ where $U=\bigsqcup_{i=1}^mZ(\alpha_i, g_i, \beta_i)$.
\end{cor}
\noindent{\bf  Questions}. When is $\ldbrack\Gg(G,E)\rdbrack$ finitely generated? When is of type $F_\infty$? When does it have the Haagerup property?

\bigskip

\begin{rmk}
For a pseudo free self-similar action $(G,E)$ with Exel-Pardo groupoid $\Gg(G,E)$, there is a cocycle $\rho: \Gg(G,E)\to \ZZ$ given by $[\alpha, g, \beta; \xi]\mapsto |\alpha|-|\beta|$ with kernel
\[\Hh(G,E)=\{[\alpha, g, \beta; \xi]\in \Gg(G,E) :|\alpha|=|\beta|\}.\] 
The cocyle is well-defined since it does not depend on germ representatives.
Now $\Hh(G,E)=\bigcup_{k\ge 0} \Hh_k(G,E)$ where
\[\Hh_k(G,E)=\{[\alpha, g, \beta; \xi]\in \Gg(G,E) :|\alpha|=|\beta|=k\}.\]
There are groupoid homomorphisms \[\omega_k:\Hh_k(G,E)\to G,\; \omega_k([\alpha,g,\beta;\xi])=g\] and $\ker\omega_k$ is $AF$ for all $k\ge 0$.

\end{rmk}

\begin{proof}
 Indeed, consider $R_k$ the equivalence relation on $E^k$ such that $(\alpha, \beta)\in R_k$ if there is $g\in G$ with $g\cdot s(\beta)=s(\alpha)$. Then the map $[\alpha, g, \beta;\xi]\mapsto ((\xi, g),(\alpha, \beta))$ gives an isomorphism between $\Hh_k(G,E)$ and $(E^\infty\rtimes G)\times R_k$, so $\ker\omega_k$ is isomorphic to $E^\infty\times R_k$. 
Since the groupoids $\Hh_k(G,E)\times_{\omega_k}G$ and $\ker \omega_k$ are similar, we obtain \[H_i(\Hh_k(G,E)\times_{\omega_k}G)\cong H_i(\ker \omega_k)=0\] for $i\ge 1$. 
\end{proof}
\begin{cor}
We have 
 \[ H_i(\Hh(G,E))=\varinjlim H_i(\Hh_k(G,E)),\] where the inclusion maps are given by
\[\chi_{Z(\alpha, g,\beta;Z(\beta\gamma))}\mapsto \chi_{Z(\alpha(g\cdot \gamma),g|_\gamma,\beta\gamma;Z(\beta\gamma))}.\]
Using the notation from a previous section, we have \[C^*(G,E)^{\TT}=\Ff(G,E)\cong C^*(\Hh(G,E)),\;\;\Ff_k\cong C^*(\Hh_k(G,E)).\] We also have
\[\bigcup_{k\ge 0}\ldbrack \Hh_k(G,E)\rdbrack=\ldbrack \Hh(G,E)\rdbrack\subseteq \ldbrack \Gg(G,E)\rdbrack\]
by considering $G$-tables where $|\alpha_i|=|\beta_i|$.
\end{cor}
\begin{rmk}
Assuming that the groupoid $G$ is transitive, it follows that $\Gg(G,E)$ is minimal. Then $G$ is Morita equivalent to the group $G_u^u$ for some $u\in E^0$, and $\Hh_k(G,E)$ is Morita equivalent to $E^\infty\rtimes G$. We deduce that
\[H_q(\Hh_k(G,E))\cong H_q(E^\infty\rtimes G)\cong H_q(G,C(E^\infty,\ZZ))\cong H_q(G_u^u, C(uE^\infty,\ZZ))\]
and 
\[H_q(\Hh(G,E))\cong\varinjlim H_q(G_u^u, C(uE^\infty,\ZZ)).\]

Recall that the boundary map $\delta_1:C_c(\Hh(G,E))\to C_c(\Hh(G,E)^{(0)})$ is given by
\[\delta_1(\chi_{Z(\alpha, g, \beta)})= \chi_{Z(\beta, s(\beta),\beta)}-\chi_{Z(\alpha, s(\alpha),\alpha)}\]
for $\alpha, \beta\in E^*$ with $|\alpha|=|\beta|$ and $g\cdot s(\beta)=s(\alpha)$. It follows that 
\[[\chi_{Z(\alpha, s(\alpha), \alpha)}]=[\chi_{Z(\beta, s(\beta), \beta)}]\in H_0(\Hh(G,E))\]
whenever $(\alpha, \beta)\in R_k$ for some $k$. Hence for  $G$  transitive, the group $H_0(\Hh(G,E))$ is generated by $[\chi_{Z(\alpha, u, \alpha)}]$ where $\alpha\in E^*u$.

The boundary map $\delta_2:C_c(\Hh(G,E)^{(2)})\to C_c(\Hh(G,E))$ is given by
\[\delta_2(\chi_{Z(\alpha, g,\beta)\times Z(\beta, h, \gamma)})= \chi_{Z(\beta, h,\gamma)}-\chi_{Z(\alpha, gh,\gamma)}+ \chi_{Z(\alpha, g,\beta)}\]
for $\alpha, \beta, \gamma\in E^*$ with $|\alpha|=|\beta|=|\gamma|$ and $g\cdot s(\beta)=s(\alpha)$, $h\cdot s(\gamma)=s(\beta)$. It follows that in $H_1(\Hh(G,E))=\ker \delta_1/\text{im } \delta_2$ we have
\[[\chi_{Z(\alpha, g,\beta)}]+[\chi_{Z(\beta,h,\gamma)}]=[\chi_{Z(\alpha,gh,\gamma)}].\]
Taking $h=d(g)=s(\beta)$ and $\gamma=\beta$ we get
\[[\chi_{Z(\alpha, g,\beta)}]=[\chi_{Z(\alpha, g,\beta)}]+[\chi_{Z(\beta, s(\beta), \beta)}],\]
so $[\chi_{Z(\beta, s(\beta),\beta)}]=0$ in $H_1(\Hh(G,E))$. 
Taking $h=g^{-1}$ and $\beta=\gamma$ we get
\[[\chi_{Z(\alpha, g,\beta)}]=-[\chi_{Z(\beta,g^{-1},\alpha)}].\]
If $s(\alpha)=s(\beta)=u$ and $|\alpha|=|\beta|$, we obtain for $g\in G_u^u$
\[[\chi_{Z(\alpha, g,\beta)}]=[\chi_{Z(\alpha, g,\alpha)}]+[\chi_{Z(\alpha, u,\beta)}],\]
\[[\chi_{Z(\beta, g,\alpha)}]=[\chi_{Z(\beta, g,\beta)}]+[\chi_{Z(\beta, u,\alpha)}],\]
so
\[[\chi_{Z(\alpha, g,\alpha)}]=[\chi_{Z(\alpha, g,\beta)}]-[\chi_{Z(\alpha, u,\beta)}]=-[\chi_{Z(\beta, g^{-1},\alpha)}]+[\chi_{Z(\beta, u,\alpha)}]=\]\[=-[\chi_{Z(\beta, g^{-1},\beta)}]-[\chi_{Z(\beta, u,\alpha)}]+[\chi_{Z(\beta, u,\alpha)}]=[\chi_{Z(\beta, g,\beta)}].\]
For $g,h\in G_{s(\alpha)}^{s(\alpha)}$ we have
\[[\chi_{Z(\alpha, g, \alpha)}]+[\chi_{Z(\alpha, h, \alpha)}]=[\chi_{Z(\alpha, gh, \alpha)}].\]

 \end{rmk}
\begin{prop}
Let $(G,E)$ be a self-similar action with $G$ transitive such that $\Gg(G,E)$ is  effective. If the homology group  $H_1(\Hh(G,E))$ is spanned by $[\chi_{Z(\alpha, g, \alpha)}]$ where $g\in G_{s(\alpha)}^{s(\alpha)}$, then the index map $I_\Hh:\ldbrack\Hh(G,E)\rdbrack\to H_1(\Hh(G,E))$ is surjective.
\end{prop}
\begin{proof}
Let $\alpha\in E^*$ and consider the bisection $V=Z(\alpha, g,\alpha)$ where $g\in G_{s(\alpha)}^{s(\alpha)}$. Since $d(V)=t(V)=Z(\alpha)$, we can define a full bisection $U=V\sqcup (E^\infty\setminus Z(\alpha))\subseteq \Hh(G,E)$ and we have $I_\Hh(\pi_U)=[\chi_V]$ by using Lemma 7.3 in \cite{M12}. The result follows since the elements $[\chi_V]$ span $H_1(\Hh(G,E))$.
\end{proof}

The  AH conjecture was verified  for the Katsura-Exel-Pardo groupoids by P. Nyland and E. Ortega in \cite{NO2}. The groupoid $\Gg(G,E)$ obtained from Example \ref{ka} is one the these. One of the main ingredients of the proof is to show the property TR by using a decomposition \[\ldbrack\Gg(G,E)\rdbrack=\ldbrack\Hh(G,E)\rdbrack\ldbrack\Gg_E\rdbrack.\]
\begin{prop}
The groupoid $\Gg(G,E)$ obtained  from the self-similar action in Example \ref{forest} does not satisfy the property TR.
\end{prop}

\begin{proof}
Recall that from \cite{D} it follows that $H_1(\Gg(G,E))\cong 0$, so the kernel of the index map $I: \ldbrack\Gg(G,E)\rdbrack\to H_1(\Gg(G,E))$ is the whole group $\ldbrack\Gg(G,E)\rdbrack$.
Consider the $G$-table 
\[\tau=\left(\begin{array}{cccc} u & v& e_4&e_5\\a^{-1}&cb&cb&a\\e_4&e_5&v&u\end{array}\right)\]
which determines an element $\bar{\tau}\in \ldbrack\Gg(G,E)\rdbrack$. Consider the transpositions
\[\tau_1=\left(\begin{array}{cccc} u & e_4&e_5 &v\\a^{-1}&a&v&v\\e_4&u&e_5&v\end{array}\right), \tau_2=\left(\begin{array}{cccc}v &e_5& e_4&u\\cb&b^{-1}c^{-1}&v&u\\e_5&v&e_4&u\end{array}\right),\]\[ \tau_3=\left(\begin{array}{ccc} u & v&w\\a^{-1}cb&b^{-1}c^{-1}a&w\\v&u&w\end{array}\right).\]
A computation shows that 
\[\tau_3\tau_2\tau_1\tau=\left(\begin{array}{cccc} e_4 &e_5& v&u\\v&v&v&(a^{-1}cba)^2\\e_4&e_5&v&u\end{array}\right),\]
which is not a product of transpositions, since $G_u^u\cong \ZZ$. In particular, $\tau$ is not a product of transpositions.
\end{proof}
\begin{rmk}
The groupoid $\Gg(G,E)$ obtained  from the self-similar action in Example \ref{forest} satisfies the HK conjecture, see \cite{D}. The AH-conjecture may well be true for this groupoid, but the technique of proof will be different from the one in \cite{NO2}. In particular, we would obtain that $\ldbrack\Gg(G,E)\rdbrack_{ab}=0$ and that $\ldbrack\Gg(G,E)\rdbrack$ is perfect.
\end{rmk}


\bigskip

\bigskip
\begin{thebibliography}{0000}

\bigskip




\bibitem{CM} M. Crainic and I. Moerdijk, {\em A homology theory for \'etale groupoids}, J. Reine Angew.
Math. 521 (2000), 25--46. 

\bibitem{D} V. Deaconu, {\em On  groupoids and $C^*$-algebras from  self-similar actions},  New York J. of Math. 27 (2021), 923--942. 

\bibitem{De} V. Deaconu, {\em Groupoid actions on  $C^*$-correspondences}, New York J. of Math. 24 (2018), 1020--1038.


\bibitem{DGMW} R. J. Deeley, M. Goffeng, B. Mesland, M.F.  Whittaker, {\em Wieler solenoids, Cuntz-Pimsner algebras and $K$-theory}, Ergodic Theory Dynam. Systems 38 (2018), no. 8, 2942--2988.

 
 \bibitem{EP} R. Exel, E. Pardo, {\em Self-Similar graphs: a unified treatment of Katsura and Nekrashevych algebras}, Adv. Math. 306(2017), 1046--1129.
 
 \bibitem{GPS}    T. Giordano, I. F. Putnam and C. F. Skau, {\em Full groups of Cantor minimal systems}, Israel J. Math. 111 (1999), 285--320.
    \bibitem{LRRW} M. Laca, I. Raeburn, J. Ramagge, M. F. Whittaker, {\em Equilibrium states on operator algebras associated to self-similar actions of groupoids on graphs},  Adv. Math. 331 (2018), 268--325.
    
         

 
\bibitem{MM} K. Matsumoto, H.  Matui, {\em Full groups of Cuntz-Krieger algebras and Higman-Thompson groups}. Groups Geom. Dyn. 11 (2017), no. 2, 499--531.

\bibitem{M06} H. Matui, {\em Some remarks on topological full groups of Cantor minimal systems}, International J. Math. vol 17 no.2 (2006) 231--251.

\bibitem{M12} H. Matui, {\em Homology and topological full groups of \'etale groupoids on totally disconnected spaces}, Proc. London Math. Soc. (3) 104 (2012),  27--56.


\bibitem{M15} H. Matui, {\em Topological full groups of one-sided shifts of finite type}, J. Reine Angew.
Math. 705 (2015),  35--84.

\bibitem{M17} H. Matui, {\em Topological full groups of \'etale groupoids}, Operator algebras and applications--the Abel Symposium 2015, 203--230, Springer 2017.

\bibitem{N04} V. V. Nekrashevych, {\em Cuntz-Pimsner algebras of group actions}, J. Operator Theory, 52 (2004),  223--249.
 
 \bibitem{N} V. Nekrashevych, {\em Self-similar groups}, Math. Surveys Monogr. 117, AMS Providence 2005.

\bibitem{N09} V. Nekrashevych, {\em $C^*$-algebras and self-similar groups}, J. Reine Angew. Math. 630 (2009) 59--123.

\bibitem{N19} V. Nekrashevych, {\em Simple groups of dynamical origin}, Ergodic Theory Dynam. Systems, 39(3): 707--732, 2019.

\bibitem{NO} P. Nyland, E. Ortega, {\em Topological full groups of ample groupoids with applications to graph algebras}, 
Internat. J. Math. 30 (2019), no. 4.

\bibitem{NO2} P. Nyland, E. Ortega, {\em Katsura-Exel-Pardo groupoids and the AH conjecture}, arXiv: 2007.06638v1 [math.OA].

\bibitem{O20} E. Ortega, {\em The homology of the Katsura-Exel-Pardo groupoid}, J. Noncommut. Geom. 14(2020), 913--935.
   

\end{thebibliography}
\end{document}